\documentclass[journal]{IEEEtran}
\usepackage[numbers,sort&compress]{natbib}

\usepackage{algorithm}

\usepackage{pgf}
\usepackage{tikz}
\usetikzlibrary{arrows, decorations.pathmorphing, backgrounds, positioning, fit, petri, automata}
\definecolor{yellow1}{rgb}{1,0.8,0.2}

\usepackage{makecell}

\usepackage{authblk}
\usepackage[T1]{fontenc}
\usepackage{inputenc}

\usepackage{amsmath,amsthm,amsfonts,amssymb,amsbsy,amsopn,amstext}
\usepackage{graphicx,color}
\usepackage{mathbbol,mathrsfs}
\usepackage{float,ccaption}
\usepackage{indentfirst}
\usepackage{subfigure}

\usepackage{booktabs}
\usepackage{threeparttable}
\usepackage{multirow}
\usepackage{diagbox}

\ifCLASSINFOpdf
\else
\fi
\newtheorem{thm}{Theorem}

\newtheorem{lem}{Lemma}

\newtheorem{ass}{Assumption}
\renewcommand{\P}{\ensuremath{\mathrm{P}}} 
\newcommand{\Cov}{\ensuremath{\mathrm{Cov}}} 
\newcommand{\define}{:=}
\newcommand{\A}{\mathbf{A}}
\newcommand{\Z}{\mathbf{B}}
\newcommand{\X}{\overline{c}}

\newcommand{\al}{\ensuremath{\alpha}}

\begin{document}
%
\title{Asymptotic Properties of $\mathcal{S}$-$\mathcal{AB}$ Method  with Diminishing Step-size}
%
%
%

\author{Shengchao Zhao,~\IEEEmembership{}
        Yongchao Liu,~\IEEEmembership{}
\thanks{School of Mathematical Sciences, Dalian University of Technology, Dalian 116024, China, e-mail: zhaoshengchao@mail.dlut.edu.cn (Shengchao Zhao), lyc@dlut.edu.cn (Yongchao Liu) }}
\maketitle

\begin{abstract}
The popular $\mathcal{AB}$/push-pull method for distributed optimization problem  may unify much of the existing decentralized first-order methods  based on gradient tracking technique. More recently, the stochastic gradient variant
of $\mathcal{AB}$/Push-Pull method ($\mathcal{S}$-$\mathcal{AB}$) has been proposed, which achieves  the linear  rate of converging to  a  neighborhood  of the  global minimizer when the step-size is constant.  
This paper is devoted to the  asymptotic properties of  $\mathcal{S}$-$\mathcal{AB}$ with diminishing step-size. Specifically,
under the condition that each local objective is smooth and the global objective is strongly-convex, we first present the boundedness of the iterates of $\mathcal{S}$-$\mathcal{AB}$  and then  show  that  the iterates converge to the global minimizer with the rate $\mathcal{O}\left(1/k\right)$ in the mean square sense. Furthermore, the asymptotic normality of Polyak-Ruppert averaged $\mathcal{S}$-$\mathcal{AB}$ is
obtained and applications on statistical inference are discussed. Finally,  numerical tests are  conducted to demonstrate the theoretic results.
\end{abstract}


%
\IEEEpeerreviewmaketitle

\section{Introduction}
%
%
%
%
\IEEEPARstart{D}{istributed} optimization problem is to minimize a finite sum
of functions over a network made up of multi agents, where each agent holds the local objective function and can communicate with its neighbors.  Distributed optimization problem has many applications  in large-scale machine learning \cite{boyd2011}, wireless networks \cite{naghshineh1996distributed}, parameter estimation \cite{Towfic2015}, to name a few. Over the last decades, numerous algorithms for  distributed optimization problem have been developed. This includes (sub)gradient method \cite{pu2021sharp,bianchi2013performance,Mor2012dsa,omidvar2020hybridorder,swenson2020distributed,fallah2019robust,de2017variance}, dual averaging method \cite{Tsi2012dual,DC2012}, primal dual method \cite{Lei2018,yi2020primaldual}, gradient push method \cite{Nedic2016sGpush,pmlr-v97-assran19a,qureshi2020pushsaga}, gradient tracking method \cite{pu2020distributed,Xin2012improv,Xin_2020}.  We refer  to the survey \cite{YANG2019278} for the new development on distributed optimization.

Recently,  Xin and Khan \cite{Xin2018linear} and  Pu et al. \cite{pu2018push}  propose  the $\mathcal{AB}$/push-pull method, which is suitable for  the  cases that the communication network is directed.  
$\mathcal{AB}$/push-pull method  utilizes a  row  stochastic  matrix  for mixing the  decision variables and a  column  stochastic  matrix  for tracking the average gradients.  
Since  $\mathcal{AB}$/push-pull method may unify much of the existing decentralized first-order methods  based on gradient tracking technique \cite{RX2020}, many works have been devoted to the development of $\mathcal{AB}$/push-pull method.  Pu \cite{pu2020robust} introduces an adapted version of push-pull method which inherits the linear convergence property of push pull method under noiseless communication links and has more robust performance than push pull method under noisy information exchange.  Xiong et al. \cite{xiong2021quantized} propose a novel quantized distributed gradient tracking method to improve communication efficiency further, which can be rewritten as an inexact version of $\mathcal{AB}$/push-pull method.
Saadatniaki et al. \cite{saadatniaki2018optimization} propose a variant of $\mathcal{AB}$/push-pull method with time-varying weight matrices and show that the proposed method converges linearly to the optimal solution when each local objective is smooth and the global objective is strongly-convex. Accelerate techniques have also been incorporated into  $\mathcal{AB}$/push-pull method, Xin and Khan \cite{Xin2020heavy} employ the heavy-ball method to accelerate the  $\mathcal{AB}$/push-pull method  where the R-linear rate for strongly convex smooth objective has been obtained.  Xin et al. \cite{RX2019Nestrov} combine  nesterov gradient method with $\mathcal{AB}$/push-pull method and show that the new method achieves robust numerical performance  for both convex and strongly-convex objectives. Moreover, the extended versions of $\mathcal{AB}$/push-pull method have been used to solve  resource allocation \cite{zhang2020dual} and the distributed multi-cluster game \cite{zimmermann2021gradienttracking}.

More recently,  stochastic gradient version of $\mathcal{AB}$/push-pull method ($\mathcal{S}$-$\mathcal{AB}$) has been proposed in \cite{RX2019,pu2020distributed}.
Xin et al. \cite{RX2019} focus on the case that the step-size is constant and show that  the iterates of  $\mathcal{S}$-$\mathcal{AB}$ converge linearly  to  a  neighborhood  of  the  global minimizer when agents' cost functions are smooth and
strongly-convex.
Pu and Nedi\'{c} \cite{pu2020distributed} study the convergence properties of distributed stochastic gradient tracking method (DSGT), which can be regarded as a special $\mathcal{S}$-$\mathcal{AB}$ with doubly stochastic weight matrix. The authors show that  DSGT  converges to the global optimal solution with the optimal rate $\mathcal{O}(1/k)$ when the stepsize diminishes to zero.  
Motivated by \cite{RX2019,pu2020distributed}, this paper is devoted to developing $\mathcal{S}$-$\mathcal{AB}$ by  studying 
the convergence properties when the step-size is
diminishing.
We focus on the convergence of the iterate of $\mathcal{S}$-$\mathcal{AB}$ in the
mean square sense and its asymptotic normality. As far as we
are concerned, the contributions of the paper can be summarized as follows.
\begin{itemize}
	\item [$\bullet$] The optimal   convergence rate of $\mathcal{S}$-$\mathcal{AB}$ in the mean square sense.   We show that the iterates of $\mathcal{S}$-$\mathcal{AB}$ converge to the global optimal solution with the rate $\mathcal{O}(1/k)$ when each local objective is smooth and the global objective is strongly-convex, which arrives at the optimal rate of vanilla stochastic gradient decent algorithm with diminishing stepsize.
	Indeed,  the convergence of stochastic approximation (SA) based methods for distributed  stochastic optimization problem with diminishing step-size have been well studied \cite{bianchi2013performance,Lei2018,pu2020distributed,Nedic2016sGpush}, where most of the results are based on the fact that the  stochastic gradient noise  is a  martingale difference  sequence. Note that 
	agents' accumulated stochastic gradient noises during gradient tracking steps
	form the autoregressive moving average processes \cite{chen2006stochastic}. The challenge for $\mathcal{S}$-$\mathcal{AB}$ is that
	the weighted average of the accumulated stochastic gradient noises is not a martingale difference sequence as the weight matrices are not doubly
	stochastic.
	

	\item [$\bullet$] The asymptotic normality of $\mathcal{S}$-$\mathcal{AB}$. We present that Polyak-Ruppert averaged $\mathcal{S}$-$\mathcal{AB}$ converges in distribution to a
	normal random vector for any agent
	by combining the convergence rate of $\mathcal{S}$-$\mathcal{AB}$ and  results on asymptotic normality of
	SA based algorithm \cite[Theorem 3.4.2]{chen2006stochastic}.
	The asymptotic normality  of   stochastic approximation based algorithms can be traced back to 1950s \cite{chung1954stochastic,fabian1968asymptotic}.
	For the asymptotic normality on distributed stochastic optimization problem, we refer to \cite{bianchi2013performance,Mor2012dsa,Lei2018} for unconstrained problems and \cite{Sahu2016,zhao2020asymptotic} for constrained problems.
	Complement to the works mentioned above, the new result does not need  the communication network to be undirected or weight matrix to be doubly
	stochastic.
	\item [$\bullet$] The estimator of the covariance matrix. A series of works \cite{chen2020statistical, hsieh2002confidence, Hao2021sta} have focused on   estimating the covariance matrix in the limit normal distribution of vanilla stochastic gradient decent algorithm. We extend the plug-in method \cite{chen2020statistical}   to estimate the  covariance matrix for the normal distribution of $\mathcal{S}$-$\mathcal{AB}$, which paves the way for doing statistical inference.
	Moreover,  numerical experiments are provided to support our theoretical analysis.
	
	%
\end{itemize}

The rest of this paper is organized as follows. Section \ref{mod-ass} introduces the distributed stochastic optimization problem model, the $\mathcal{S}$-$\mathcal{AB}$ method and presents some standard assumptions in distributed stochastic optimization problem. Section \ref{conver-rate} focuses on the convergence rate and asymptotic normality of the $\mathcal{S}$-$\mathcal{AB}$ method.  At last,
numerical results are presented in Section \ref{num-exm}
to validate the theoretic results.

Throughout this paper, we use the following notations. $\mathbb{R}^d$ denotes  the d-dimension Euclidean space endowed with norm $\|x\|=\sqrt{\langle x,x\rangle}$.  Denote $\mathbf{1}:=(1~1\dots1)^\intercal$, $\mathbf{0}\define(0~0\dots0)^\intercal$ and whose dimension will be clear in the context.
$\mathbf{I}_{d}\in\mathbb{R}^{d\times d}$ stands for the identity matrix.
$\mathbf{A}\otimes \mathbf{B}$ denotes the Kronecker product of matrices $\mathbf{A}$ and $\mathbf{B}$.
For a sequence of random vectors $\{\xi_k\}$ and a random vector $\xi$, $\xi_k\stackrel{d}{\rightarrow} \xi$ represents the convergence in distribution, \Cov($\xi$) denotes the covariance matrix of random vector $\xi$. $N\left(\mu,\Sigma\right)$ is the normal distribution with mean $\mu$ and covariance matrix $\Sigma$. For any sequences $\{a_k\}$ and $\{b_k\}$ of positive number,  $a_k=\mathcal{O}(b_k)$ if there exists $c>0$ such that $a_k\le c b_k$.

\section{Problem formulation and $\mathcal{S}$-$\mathcal{AB}$ method}\label{mod-ass}

In this paper, we consider the following distributed stochastic optimization  (DSO) problem
\begin{equation}
\begin{aligned}\label{problem model}
&\min_{x\in\mathbb{R}^d} ~f(x)=\sum_{j=1}^nf_j(x),
\end{aligned}
\end{equation}
where $f_j(x):=\mathbb{E}[g_j(x;\zeta_j)]$ denotes the cost  of agent $j$, $g_j(x;\zeta_j):\mathbb{R}^d\times \Omega\rightarrow \mathbb{R}$ is the measure function, $\zeta_j$ is a random variable defined on a probability space $(\Omega,\mathcal{F},\P)$ and $\mathbb{E}[\cdot]$ denotes the expectation with respect to probability $\P$. For DSO (\ref{problem model}),
the communication relationship between agents is characterized by a directed graph, $\mathcal{G}=\left(\mathcal{V},\mathcal{E}\right)$, where $\mathcal{V}=\{1,2,...,n\}$ is the node set with node $i$ representing agent $i$ and $\mathcal{E}\subseteq\mathcal{V}\times\mathcal{V}$ denotes the edge set with edge $(j,i)$ representing agent $i$ can receive information from agent $j$.

The  $\mathcal{S}$-$\mathcal{AB}$ method in \cite{RX2019} reads as follows.
\begin{algorithm}[H]
	\caption{$\mathcal{S}$-$\mathcal{AB}$: At each node $i\in\mathcal{V}=\{1,2,...,n\}$}\label{alg:SAB}
	\textbf{Require:} initial value $x_{i,0}\in\mathbb{R}^{d}$, $y_{i,0}=\nabla g_i(x_{i,0},\zeta_{i,0})$, weight matrices $\mathbf{A}=\{a_{ij}\}$ and $\mathbf{B}=\{b_{ij}\}$.
	\begin{itemize}
		\item[1:]\textbf{For} $k=1,2,\cdots$ \textbf{do}
		\item [2:]\textbf{State update:}
		\begin{equation}\label{alg:x}
		x_{i,k+1}=\sum_{j=1}^n a_{ij} x_{j,k}-\alpha_k y_{i,k}.
		\end{equation}
		\item[3:]\textbf{Gradient tracking update:}
		$$
		 y_{i,k+1}=\sum_{j=1}^n b_{ij} y_{j,k}+\nabla g_i(x_{i,k+1},\zeta_{i,k+1})-\nabla g_i(x_{i,k},\zeta_{i,k}),
		$$
		where $\zeta_{i,0},\zeta_{i,1},\cdots$  is independent and identically distributed sample of $\zeta_{i}$.
		\item[4:]\textbf{end for}
	\end{itemize}
\end{algorithm}

Throughout our analysis in the paper, we make the following two assumptions on the objective function and weight matrices $\A$ and $\Z$. 

\begin{ass}[\textbf{objective function}]\label{ass:objective}
	\begin{itemize}
		\item [(i)]Function $f(x)$ is differentiable and $\mu$-strongly convex in $x$, that is
		\begin{equation*}
		f(y)\ge f(x)+\langle \nabla f(x),y-x\rangle+\frac{\mu}{2}\|x-y\|^2, \quad \forall x,y\in\mathbb{R}^d.
		\end{equation*}
		\item [(ii)] Let $x^*$ be an optimal solution. There exists a constant $c>0$ such that
	\begin{equation*}
		\|\nabla f(x)-\nabla^2 f(x^*)\left(x-x^*\right)\|\le c\|x-x^*\|^2,\quad \forall x\in\mathbb{R}^n.
		\end{equation*}
		\item [(iii)] For any $i\in\mathcal{V}$, function $g_i(x,\zeta_{i})$ is differentiable and there exist constants $p>2$, $U_1>0$ such that
		\begin{equation*}
		\begin{aligned}
		&\mathbb{E}\left[\|\nabla g_i(x;\zeta_{i})-\nabla f_i(x)\|^{p}\right]\le U_1^{p/2},\\
		&\|\nabla g_i(x;\zeta_i)-\nabla g_i(y;\zeta_i)\|\le L_i(\zeta_i)\|x-y\|, \quad \forall x,y\in\mathbb{R}^d,
		\end{aligned}
		\end{equation*}
		where $L_i(\zeta_i)$ is a positive random variable  and  satisfies $\mathbb{E}[L_i^p(\zeta_i)]<\infty$.
	\end{itemize}
	
\end{ass}

 Assumption \ref{ass:objective} (i)- (ii) are standard conditions on objective  and have been well used to study the  asymptotic properties of SA based method \cite{duchi2016asymptotic,Polyak1992}. 
In Assumption \ref{ass:objective} (iii),  the constant $p=2$ is sufficient to investigate the convergence rate of $\mathcal{S}$-$\mathcal{AB}$ as \cite{RX2019,RX2020},  while   $p>2$ is needed for   the  asymptotic normality of $\mathcal{S}$-$\mathcal{AB}$. Define $L:=\max_{1\leq i \leq n}\sqrt{\mathbb{E}[L_{i}^2(\zeta_i)]}$. Then for any $x,y\in\mathbb{R}^d$, we have 
\begin{align}\label{ie-1}
\|\nabla f_i(x)-\nabla f_i(y)\|^2
&\le L^2\|x-y\|^2
\end{align} 
and 
{\small	\begin{align}\label{ie-2}
	\left\|\frac{1}{n}\sum_{j=1}^n\nabla f_j(x_i)-\frac{1}{n}\nabla f(y)\right\|^2&=\left\|\frac{1}{n}\sum_{j=1}^n\left(\nabla f_j(x_i)-\nabla f_j(y)\right)\right\|^2\notag\\
	&\le \frac{1}{n}\sum_{j=1}^n\left\|\nabla f_j(x_i)-\nabla f_j(y)\right\|^2\notag\\
	&\le \frac{L^2}{n}\sum_{j=1}^n\|x_i-y\|^2.
	\end{align}}

\begin{ass}  [\textbf{weight matrices and networks}]\label{ass:matrix}
	Let $\mathcal{G}_A=\left(\mathcal{V},\mathcal{E}_\A\right)$ and $\mathcal{G}_{B^\intercal}=\left(\mathcal{V},\mathcal{E}_{\Z^\intercal}\right)$ be graphs induced by  matrices $\mathbf{A}$ and $\mathbf{B}^\intercal$ respectively.  Suppose that  
	\begin{itemize}
		\item [(i)]The matrix $\mathbf{A}\in\mathbb{R}^{n\times n}$ is nonnegative row stochastic and $\mathbf{B}\in\mathbb{R}^{n\times n}$ is nonnegative column stochastic, i.e., $\mathbf{A}\mathbf{1}=\mathbf{1}$ and $\mathbf{1}^\intercal\mathbf{B}=\mathbf{1}^\intercal$. In addition, the diagonal entries of $\mathbf{A}$ and $\mathbf{B}$ are positive, i.e., $\mathbf{A}_{ii}>0$ and $\mathbf{B}_{ii}>0$ for all $i\in \mathcal{V}$.
		\item[(ii)]  The graphs $\mathcal{G}_A$ and $\mathcal{G}_{B^\intercal}$ each contain at least one spanning tree. Moreover, there exists at least one node that is a root of spanning trees for both $\mathcal{G}_A$ and $\mathcal{G}_{B^\intercal}$, i.e. $\mathcal{R}_A\cap \mathcal{R}_{B^\intercal}\ne\emptyset$, where $\mathcal{R}_A$ ( $\mathcal{R}_{B^\intercal}$) is the set of roots of all possible spanning trees in the graph $\mathcal{G}_A$ ( $\mathcal{G}_{B^\intercal}$).
	\end{itemize}
\end{ass}

As commented in \cite{pu2020push}, Assumption \ref{ass:matrix} is weaker than requiring that both $\mathcal{G}_A$ and $\mathcal{G}_{B^\intercal}$
are strongly connected.
Building on consensus with  non-doubly stochastic matrices, the scheme of $\mathcal{AB}$/push-pull method has been applied on reinforce learning \cite{REN2021RL} and economic dispatch problem \cite{WU2022economic}. Under  Assumption \ref{ass:matrix}, the matrix $\mathbf{A}$ has a nonnegative left eigenvector $u^\intercal$ (w.r.t. eigenvalue 1) with $u^\intercal\mathbf{1}=n$, and the matrix $\mathbf{B}$ has a nonnegative right eigenvector $v$ (w.r.t. eigenvalue 1) with $\mathbf{1}^Tv=n$. Moreover, $u^\intercal v>0$.

	For easy of presentation, we rewrite Algorithm \ref{alg:SAB} in a compact form:
	\begin{equation}\label{alg:new form}
	\begin{aligned}
	&x_{k+1}=\tilde{\mathbf{A}}x_k-\alpha_ky_k,\\
	&y_{k+1}=\tilde{\mathbf{B}}y_k+\nabla G_{k+1}-\nabla G_k,
	\end{aligned}
	\end{equation}
	where $\tilde{\mathbf{A}}\define\mathbf{A}\otimes \mathbf{I}_d,~ \tilde{\mathbf{B}}\define\mathbf{B}\otimes \mathbf{I}_d$, the vectors $x_k$, $y_k$ and $\nabla G_k$ concatenate all $x_{i,k}$'s, $y_{i,k}$'s and $\nabla g_i(x_{i,k},\zeta_{i,k})$'s respectively.

\section{The convergence rate and asymptotic normality}\label{conver-rate}
In this section, we study the convergence  rate and asymptotic normality of $\mathcal{S}$-$\mathcal{AB}$ method. 
As commented in the introduction, the convergence rate of $\mathcal{S}$-$\mathcal{AB}$ has been discussed in \cite{RX2019,pu2020distributed}. We focus on the rate of $x_{i,k}\to x^*$ rather than  $x_{i,k}$ converging into
a neighbor of $x^*$ \cite{RX2019}. Moreover, different from \cite{pu2020distributed},  the weighted average of the accumulated noises among agents can not form the martingale difference sequence as weight matrices $\mathbf{A}$ and $\mathbf{B}$ are not doubly stochastic.

We first study the    convergence rate  of iterates of $\mathcal{S}$-$\mathcal{AB}$ in the mean square sense. The following technical Lemma is  an extension of   \cite[Lemma 3]{bianchi2013performance}, which plays a key role in studying the
stability and agreement of  $\mathcal{S}$-$\mathcal{AB}$ method.

\begin{lem}\label{lem:ieq-1}
	Suppose that positive sequences $\{\gamma_{k}\}$,  $\{\rho_{k}\}$,  $\{\phi_{k}\}$ satisfy
	\begin{itemize}
		\item[(i)]$\{\gamma_{k}\}$,  $\{\rho_{k}\}$ are $[0,1]$-valued sequences such
		that $\sum_{k=0}^\infty\gamma_{k}^2<\infty$, $\limsup_{k\rightarrow\infty}\frac{\gamma_{k}}{\gamma_{k+1}}=1$.
		\item[(ii)]	
		\begin{align*}
			&\limsup_{k\rightarrow\infty}\left(\gamma_{k}\sqrt{\phi_k}+\frac{\phi_{k-1}}{\phi_k}\right)<\infty,\quad \sum_{k=0}^\infty\phi_k^{-1}<\infty,\\
			&\liminf_{k\rightarrow\infty}\left(\gamma_{k}\sqrt{\phi_k}\right)^{-1}\left(\frac{\phi_{k-1}}{\phi_k}-\rho_k\right)>0.
			\end{align*}
	\end{itemize}
	If for scalars $M>0$ and $1>\rho>0$, sequences $\{u_k\},~\{v_k\}$ satisfy that for $k\ge k_0$,
	\begin{align}
	\label{form-1}u_{k+1}&\le\rho_ku_k+M\gamma_{k}\sqrt{u_k(1+u_k+v_k)}\\
	&\quad+M\gamma_{k}\left(\gamma_{k}+\sum_{t=k_0}^k\gamma_{t}\rho^{k-t}u_t+\sum_{t=k_0}^k\gamma_{t}\rho^{k-t}v_t\right),\notag\\
	\label{form-2}v_{k+1}&\le v_k+Mu_k+M\gamma_{k}\sqrt{u_k(1+u_k+v_k)}\\
	&+M\gamma_{k}\left(\gamma_{k}+\sum_{t=k_0}^k\gamma_{t}\rho^{k-t}u_t+\sum_{t=k_0}^k\gamma_{t}\rho^{k-t}v_t\right),\notag
	\end{align}
	then
	\begin{equation*}
	\sup_kv_k<\infty,\quad \limsup_{k\rightarrow\infty}\phi_ku_k<\infty.
	\end{equation*}
\end{lem}

With Lemma \ref{lem:ieq-1} at hand, we are  ready for
presenting the stability and agreement of $\mathcal{S}$-$\mathcal{AB}$. To facilitate analysis, we define two auxiliary sequences $\{y_k^{'}\}$ and $\{\xi_{k}\}$ as
\begin{align}
\label{real-y}&y_{k+1}^{'}=\tilde{\mathbf{B}}y_k^{'}+\nabla F_{k+1}-\nabla F_k,\quad y_{0}^{'}=\nabla F_{0},\\
\label{auxi-seq}&\xi_{k+1}=\tilde{\mathbf{B}}\xi_k+\epsilon_{k+1}-\epsilon_k,\quad \xi_{0}=\epsilon_{0},
\end{align}
where vectors $\nabla F_k$ and $\epsilon_k$ concatenates all $\nabla f_i(x_{i,k})$'s and $\epsilon_{i,k}$'s respectively, $\epsilon_{i,k}\define\nabla g_i(x_{i,k};\zeta_{i,k})-\nabla f_i(x_{i,k})$. 

\begin{lem}\label{lem:rate}
	Suppose that Assumptions \ref{ass:objective} and \ref{ass:matrix} hold. Let step-size $\alpha_{k}=a/(k+b)^\alpha$, where $\alpha\in (1/2,1]$, positive scalars $a,b$ satisfy  $\frac{a}{b^\alpha}\le \min\{1,\frac{n}{u^\intercal vL}\}$, $L\define\max_{1\leq i \leq n}\sqrt{\mathbb{E}[L_{i}^2(\zeta_i)]}$, where $u^\intercal$ and $v$ are the left eigenvector of $\mathbf{A}$ and the right eigenvector of $\mathbf{B}$ respectively (see Part \ref{mod-ass} for the definitions).  Then there exists a positive constant $U_2$ such that
	\begin{equation*}
	\sup_{k}\mathbb{E}\left[\|\bar{x}_{k}-x^*\|^2\right]\le U_2,\quad
	\mathbb{E}\left[\|x_{k}-\mathbf{1}\otimes\bar{x}_{k}\|^2\right]\le U_2\alpha_k^2
	\end{equation*}
	and
	\begin{equation*}
	\mathbb{E}\left[\|y_{k}^{'}-v\otimes\bar{y}_{k}^{'}\|^2\right]\le U_2\alpha_k^2,
	\end{equation*}
	where
	\begin{equation}\label{average}
	\bar{x}_k=\left(\frac{u^\intercal}{n}\otimes\mathbf{I}_{d}\right)x_k,\quad \bar{y}_k^{'}\define\left(\frac{\mathbf{1}^\intercal}{n}\otimes\mathbf{I}_{d}\right)y_k^{'}.
	\end{equation}
\end{lem}
Before presenting the proof of Lemma \ref{lem:rate}, we recall two vector norms
\cite[Lemma 3]{Song2021CompressedGT}
\begin{equation*}
\|\mathbf{x}\|_\mathbf{A}\define\left\|\hat{\mathbf{A}}\mathbf{x}\right\|,\quad \|\mathbf{x}\|_\mathbf{B}\define\left\|\hat{\mathbf{B}}\mathbf{x}\right\|,\quad \forall x\in\mathbb{R}^{nd}
\end{equation*}
and their induced matrices norms
\begin{equation*}
\left\|\mathbf{C}\right\|_\mathbf{A}\define\sup_{x\neq0}\frac{\|\mathbf{C}x\|_\mathbf{A}}{\|x\|_\mathbf{A}},~ \left\|\mathbf{C}\right\|_\mathbf{B}\define\sup_{x\neq0}\frac{\|\mathbf{C}x\|_\mathbf{B}}{\|x\|_\mathbf{B}},~ \forall \mathbf{C}\in\mathbb{R}^{nd\times nd},
\end{equation*} 
 where $\hat{\mathbf{A}},\hat{\mathbf{B}}\in \mathbb{R}^{nd\times nd}$ are some invertible matrices. These vector norms and their induced matrices norms have the following properties.
\begin{itemize}
	\item [(i)]
	\begin{align}
	\label{consensus-para}&\tau_\mathbf{A}\define\left\|\tilde{\mathbf{A}}-\frac{\mathbf{1}u^\intercal}{n}\otimes \mathbf{I}_d\right\|_\mathbf{A}< 1,\\ \label{consensus-para-1}&\tau_\mathbf{B}\define\left\|\tilde{\mathbf{B}}-\frac{v\mathbf{1}^\intercal}{n}\otimes \mathbf{I}_d\right\|_\mathbf{B}<1.
	\end{align}
	\item [(ii)]
	\begin{equation*}
	\|\mathbf{C}x\|_\A\le \|\mathbf{C}\|_\mathbf{A}\|x\|_\mathbf{A},\quad \|\mathbf{C}x\|_\mathbf{B}\le \|\mathbf{C}\|_\mathbf{B}\|x\|_\mathbf{B}
	\end{equation*}
	for any $\mathbf{C}\in\mathbb{R}^{nd\times nd},x\in\mathbb{R}^{nd}$.
	\item [(iii)]Let $\|\cdot\|_*$ and $\|\cdot\|_{**}$ be any two vector norms of $\|\cdot\|$, $\|\cdot\|_\A$ and $\|\cdot\|_\Z$. There exists a constant $\X>1$ such that
	\begin{equation}\label{norm-bound-1}
	\|x\|_*\le \X \|x\|_{**},\quad x\in\mathbb{R}^{nd}.
	\end{equation}
\end{itemize}

\begin{proof}
	We employ  Lemma 1 to prove Lemma \ref{lem:rate} and  finish the proof by the following two steps: find relationships of $u_k$ and $v_k$ in the forms of (\ref{form-1}) and (\ref{form-2}) first, and then verify conditions (i)-(ii) of Lemma \ref{lem:ieq-1}.
	
	\textbf{Step 1.}
	By the definitions of $\bar{x}_{k}$, $y_k^{'}$ and $\xi_{k+1}$ in (\ref{real-y})-(\ref{average}), we have $y_k=y_k^{'}+\xi_{k+1}$ and 
	\begin{equation}\label{consensus-new-form}
	\begin{aligned}
	\bar{x}_{k+1}
	&=\left(\frac{u^\intercal}{n}\otimes\mathbf{I}_{d}\right)\tilde{\mathbf{A}}x_k-\alpha_k\left(\frac{u^\intercal}{n}\otimes\mathbf{I}_{d}\right)y_k\\
	&=\bar{x}_{k}-\alpha_k\left(\frac{u^\intercal}{n}\otimes\mathbf{I}_{d}\right)\left(y_k^{'}+\xi_k\right)\\
	&=\bar{x}_{k}-\frac{u^\intercal v\alpha_{k}}{n^2}\nabla f(\bar{x}_k)-\frac{u^\intercal v\alpha_{k}}{n}\left(\bar{y}^{'}_k-\frac{1}{n}\nabla f(\bar{x}_k)\right)\\
	&\quad-\alpha_k\left(\frac{u^\intercal}{n}\otimes\mathbf{I}_{d}\right)\left(y_k^{'}-v\otimes\bar{y}^{'}_k\right)-\alpha_{k}\left(\frac{u^\intercal}{n}\otimes\mathbf{I}_{d}\right)\xi_k,
	\end{aligned}
	\end{equation}
	where the second equality follows from  $u^\intercal\mathbf{A}=\mathbf{1}$. By the recursion (\ref{consensus-new-form}),
	{\footnotesize\begin{equation}\label{eq-auxi-3-0}
		\begin{aligned}
		&\mathbb{E}\left[\|\bar{x}_{k+1}-x^*\|^2\right]\\
		&=\mathbb{E}\left[\left\|\bar{x}_{k}-x^*-\frac{u^\intercal v\alpha_{k}}{n^2}\nabla f(\bar{x}_k)\right\|^2\right]\\
		&\quad+\alpha_{k}^2\mathbb{E}\left[\left\|\frac{u^\intercal v}{n}\left(\bar{y}^{'}_k-\frac{1}{n}\nabla f(\bar{x}_k)\right)-\left(\frac{u^\intercal}{n}\otimes\mathbf{I}_{d}\right)\left(y_k^{'}-v\otimes\bar{y}^{'}_k\right)\right.\right.\\
		&\quad\left.\left.-\left(\frac{u^\intercal}{n}\otimes\mathbf{I}_{d}\right)\xi_k\right\|^2\right]-2\mathbb{E}\left[\left\langle \bar{x}_{k}-x^*-\frac{u^\intercal v\alpha_{k}}{n^2}\nabla f(\bar{x}_k),\right.\right. \\
		&\quad\left.\left.\alpha_{k}\left(\frac{u^\intercal}{n}\otimes\mathbf{I}_{d}\right)\left(y_k^{'}-\frac{v}{n}\otimes \nabla f(\bar{x}_k)+\xi_k\right)\right\rangle\right]\\
		&\le\mathbb{E}\left[\left\|\bar{x}_{k}-x^*-\frac{u^\intercal v\alpha_{k}}{n^2}\nabla f(\bar{x}_k)\right\|^2\right]\\
		&\quad+3\alpha_{k}^2\left(\frac{u^\intercal v}{n}\right)^2\mathbb{E}\left[\left\|\bar{y}^{'}_k-\frac{1}{n}\nabla f(\bar{x}_k)\right\|^2\right]\\
		&\quad +\frac{3\alpha_{k}^2\|u\|^2\X^2}{n^2}\mathbb{E}\left[\left\|y_k^{'}-v\otimes\bar{y}^{'}_k\right\|_\mathbf{B}^2\right]+\frac{3\alpha_{k}^2\|u\|^2}{n^2}\mathbb{E}\left[\left\|\xi_k\right\|^2\right]\\
		&\quad-2\mathbb{E}\left[\left\langle \bar{x}_{k}-x^*, \alpha_{k}\left(\frac{u^\intercal}{n}\otimes\mathbf{I}_{d}\right)\left(y_k^{'}-\frac{v}{n}\otimes \nabla f(\bar{x}_k)+\xi_k\right)\right\rangle\right]+\\
		&\quad 2\frac{u^\intercal v\alpha_{k}^2}{n^2}\mathbb{E}\left[\left\langle \nabla f(\bar{x}_k), \left(\frac{u^\intercal}{n}\otimes\mathbf{I}_{d}\right)\left(y_k^{'}-\frac{v}{n}\otimes \nabla f(\bar{x}_k)+\xi_k\right)\right\rangle\right].
		\end{aligned}
		\end{equation}}
	By \cite[Lemma 10]{qu2017harnessing}\footnote{By the definition of $\alpha_{k}$, $\frac{u^\intercal v\alpha_{k}}{n^2}\le\frac{u^\intercal v a}{n^2b^\alpha}\le \frac{2}{nL}$, which satisfies the condition of \cite[Lemma 10]{qu2017harnessing}.}, the first term on the right hand of (\ref{eq-auxi-3-0})
	\begin{equation}\label{ie-3}
	\left\|\bar{x}_{k}-x^*-\frac{u^\intercal v\alpha_{k}}{n^2}\nabla f(\bar{x}_k)\right\|^2\le\left(1-\frac{u^\intercal v\mu\alpha_{k}}{n^2}\right)^2\left\|\bar{x}_{k}-x^*\right\|^2.
	\end{equation}
    For the second term on the right hand side of (\ref{eq-auxi-3-0}),
    \begin{align}\label{ie-4}
    &3\alpha_k^2\left(\frac{u^\intercal v}{n}\right)^2\mathbb{E}\left[\left\|\bar{y}^{'}_k-\frac{1}{n}\nabla f(\bar{x}_k)\right\|^2\right]\notag\\
    &=3\alpha_k^2\left(\frac{u^\intercal v}{n}\right)^2\mathbb{E}\left[\left\|\frac{1}{n}\sum_{i=1}^n\nabla f_i(x_{i,k})-\frac{1}{n}\nabla f(\bar{x}_k)\right\|^2\right]\notag\\
    &\le 3\alpha_k^2\left(\frac{u^\intercal v}{n}\right)^2\frac{L^2\X^2}{n}\mathbb{E}\left[\left\|x_k-\mathbf{1}\otimes\bar{x}_k\right\|_\mathbf{A}^2\right],
    \end{align}
    where $\X$ is defined in (\ref{norm-bound-1}), the equality follows from the fact $\bar{y}^{'}_k=\frac{1}{n}\sum_{i=1}^n\nabla f_i(x_{i,k})$ and the inequality follows from (\ref{ie-2}).
	By Lemma \ref{lem:noi-bound} (i) in Appendix \ref{apd:proof of Lem-rate}, the fourth term on the right hand of (\ref{eq-auxi-3-0}) 
	\begin{equation}\label{ie-5}
	\frac{3\alpha_{k}^2\|u\|^2}{n^2}\mathbb{E}\left[\left\|\xi_k\right\|^2\right]\le 3\alpha_{k}^2\frac{\|u\|^2c_b^2U_1}{n(1-\tau_\Z)^2},
	\end{equation}
	where  $c_b=\max\left\{\X^2,\frac{\left\|\mathbf{B}-\mathbf{I}_{n}\right\|}{\tau_\Z}\X^2\right\}$.
	Substitute (\ref{ie-3})-(\ref{ie-5}) into (\ref{eq-auxi-3-0}),
	{\small\begin{equation}\label{eq-auxi-3}
		\begin{aligned}
		&\mathbb{E}\left[\|\bar{x}_{k+1}-x^*\|^2\right]\\
		&\le\left(1-\frac{u^\intercal v\mu\alpha_{k}}{n^2}\right)^2\mathbb{E}\left[\left\|\bar{x}_{k}-x^*\right\|^2\right]+3\alpha_{k}^2\mathbb{E}\left[\left(\frac{u^\intercal v}{n}\right)^2\frac{L^2\X^2}{n}\right.\\
		&\quad \left.\left\|x_k-\mathbf{1}\otimes\bar{x}_k\right\|_\mathbf{A}^2+\frac{\|u\|^2\X^2}{n^2}\left\|y_k^{'}-v\otimes\bar{y}^{'}_k\right\|_\mathbf{B}^2\right]+3\alpha_{k}^2\frac{\|u\|^2c_b^2U_1}{n(1-\tau_\Z)^2}\\
		&\quad-2\mathbb{E}\left[\left\langle \bar{x}_{k}-x^*, \alpha_{k}\left(\frac{u^\intercal}{n}\otimes\mathbf{I}_{d}\right)\left(y_k^{'}-\frac{v}{n}\otimes \nabla f(\bar{x}_k)+\xi_k\right)\right\rangle\right]+\\
		&\quad 2\frac{u^\intercal v\alpha_{k}^2}{n^2}\mathbb{E}\left[\left\langle \nabla f(\bar{x}_k), \left(\frac{u^\intercal}{n}\otimes\mathbf{I}_{d}\right)\left(y_k^{'}-\frac{v}{n}\otimes \nabla f(\bar{x}_k)+\xi_k\right)\right\rangle\right].
		\end{aligned}
		\end{equation}}	
	For the fourth term on the right hand side of (\ref{eq-auxi-3}), we have by Lemma \ref{lem:noi-bound} (ii)  (Appendix \ref{apd:proof of Lem-rate}) that
	{\small\begin{equation}\label{cross term-0}
		\begin{aligned}
		&-2\mathbb{E}\left[\left\langle \bar{x}_{k}-x^*, \alpha_{k}\left(\frac{u^\intercal}{n}\otimes\mathbf{I}_{d}\right)\left(y_k^{'}-\frac{v}{n}\otimes \nabla f(\bar{x}_k)+\xi_k\right)\right\rangle\right]\\
		&\le 2\alpha_{k}\left(\mathbb{E}\left[\left\|\bar{x}_{k}-x^*\right\|^2\right]\left(2\frac{(u^\intercal vL\X)^2}{n^{3}}\mathbb{E}\left[\left\|x_k-\mathbf{1}\otimes\bar{x}_k\right\|_\A^2\right]\right.\right.\\
		&\quad\left.\left.+2\frac{\|u\|^2\X^2}{n^2}\mathbb{E}\left[\left\|y_k^{'}-v\otimes \bar{y}^{'}_k\right\|_\Z^2\right]\right)\right)^{1/2}+\alpha_{k}^2\frac{5\|u\|^2\X c_b^2U_1 c_\rho}{n(1-\tau_\Z)^3}\\
		&\quad+ \frac{4\|u\|^2\X\alpha_k}{n^2(1-\tau_\Z)}\sum_{t=0}^{k-1}\alpha_{t}\tau_\Z^{k-t}\bigg(\mathbb{E}\left[\X^2\|y_t^{'}-v\otimes \bar{y}^{'}_t\|_\mathbf{B}^2\right]\\
		&\quad+\frac{\|v\|^2L^2\X^2}{n}\mathbb{E}\left[\|x_t-\mathbf{1}\otimes\bar{x}_t\|_\mathbf{A}^2\right]+\frac{\|v\|^2L^2}{n^2}\mathbb{E}\left[\|\bar{x}_t-x^*\|^2\right]\bigg).
		\end{aligned}
		\end{equation}}
       By the Lipschitz continuity of $\nabla f_i(\cdot)$, the fifth  term on the right hand side of (\ref{eq-auxi-3})
	{\small\begin{equation}\label{term3}
		\begin{aligned}
		&2\frac{u^\intercal v\alpha_{k}^2}{n^2}\mathbb{E}\left[\left\langle \nabla f(\bar{x}_k), \left(\frac{u^\intercal}{n}\otimes\mathbf{I}_{d}\right)\left(y_k^{'}-\frac{v}{n}\otimes \nabla f(\bar{x}_k)+\xi_k\right)\right\rangle\right]\\
		&\le 2\frac{u^\intercal v\alpha_{k}^2}{n^2}\mathbb{E}\left[ \left\|\nabla f(\bar{x}_k)\right\|\left(\frac{\|u\|}{n}\left\|y_k^{'}-\frac{v}{n}\otimes \bar{y}_k^{'}\right\|\right.\right.\\
		&\quad\left.\left.+\frac{u^\intercal v}{n}\left\|\bar{y}_k^{'}- \nabla f(\bar{x}_k)\right\|+\frac{\|u\|}{n}\left\|\xi_k\right\|\right)\right]\\
		&\le \frac{u^\intercal v\alpha_{k}^2}{n^2}\mathbb{E}\left[ 3\left\|\nabla f(\bar{x}_k)\right\|^2+
		 \left(\frac{\|u\|^2}{n^2}\left\|y_k^{'}-v\otimes \bar{y}_k^{'}\right\|^2\right.\right.\\
		 &\left.\left.+\left(\frac{u^\intercal v}{n}\right)^2\left\|\bar{y}_k^{'}-\frac{1}{n}\nabla f(\bar{x}_k)\right\|^2+\frac{\|u\|^2}{n^2}\left\|\xi_k\right\|^2\right)\right]\\
		&\le\frac{u^\intercal v\alpha_{k}^2}{n^2}\mathbb{E}\left[3n^2L^2\|\bar{x}_k-x^*\|^2+\left(\frac{u^\intercal v}{n}\right)^2\frac{L^2\X^2}{n}\left\|x_k-\mathbf{1}\otimes\bar{x}_k\right\|_\mathbf{A}^2\right]\\
		&+\frac{u^\intercal v\alpha_{k}^2}{n^2}\mathbb{E}\left[\frac{\|u\|^2\X^2}{n^2}\left\|y_k^{'}-v\otimes\bar{y}^{'}_k\right\|_\mathbf{B}^2\right]+\frac{u^\intercal v\|u\|^2c_b^2U_1}{n^3(1-\tau_\Z)^2}\alpha_{k}^2,
		\end{aligned}
		\end{equation}}
	where the first inequality follows from Cauchy-Schwartz inequality, the third inequality follows from (\ref{ie-4})-(\ref{ie-5}) and the fact 
	\begin{equation}\label{ie-6}
	\left\|\nabla f(\bar{x}_k)\right\|=\left\|\nabla f(\bar{x}_k)-\nabla f(x^*)\right\|\le nL \left\|\bar{x}_k-x^*\right\|.
	\end{equation}
	Substitute (\ref{cross term-0}) and (\ref{term3}) into (\ref{eq-auxi-3}), we have  
	\begin{equation}\label{ine-0}
	\begin{aligned}
	&\mathbb{E}\left[\|\bar{x}_{k+1}-x^*\|^2\right]\\
	&\le \left(1-\frac{u^\intercal v\mu\alpha_{k}}{n}\right)^2\mathbb{E}\left[\left\|\bar{x}_{k}-x^*\right\|^2\right]\\
	&\quad+2\alpha_{k}\left(\mathbb{E}\left[\left\|\bar{x}_{k}-x^*\right\|^2\right]\left(2\frac{(u^\intercal vL\X)^2}{n^{3}}\mathbb{E}\left[\left\|x_k-\mathbf{1}\otimes\bar{x}_k\right\|_\A^2\right]\right.\right.\\
	&\quad\left.\left.+2\frac{\|u\|^2\X^2}{n^2}\mathbb{E}\left[\left\|y_k^{'}-v\otimes \bar{y}^{'}_k\right\|_\Z^2\right]\right)\right)^{1/2}+c_0\alpha_{k}^2\\
	&\quad +c_0\alpha_{k}\sum_{t=0}^{k}\alpha_{t}\tau_\Z^{k-t}\bigg(\mathbb{E}\left[\|y_t^{'}-v\otimes \bar{y}^{'}_t\|_\mathbf{B}^2\right]\\
	&\quad+\mathbb{E}\left[\|x_t-\mathbf{1}\otimes\bar{x}_t\|_\mathbf{A}^2\right]+\mathbb{E}\left[\|\bar{x}_t-x^*\|^2\right]\bigg),
	\end{aligned}
	\end{equation}
	where $c_0>0$ is some positive constant.
	
	Note that for  any random vectors $\theta_1$, $\theta_2$ and positive scalar $\tau$, 	
	\begin{equation}\label{general-form}
	\begin{aligned}
	\mathbb{E}\left[\left|\left\|\theta_1+\theta_2\right\|\right|^2\right]
	\le (1+\tau)\mathbb{E}\left[\|\theta_1\|^2\right]+\left(1+\frac{1}{\tau}\right)\mathbb{E}\left[\|\theta_2\|^2\right],
	\end{aligned}
	\end{equation}
	where $|\|\cdot\||$ could be $\|\cdot\|_\A$ or $\|\cdot\|_\Z$.	
	Choosing
	\begin{align*}
	&\theta_1=\left(\tilde{\mathbf{A}}-\frac{\mathbf{1}u^\intercal}{n}\otimes\mathbf{I}_{d}\right)\left(x_k-\mathbf{1}\otimes\bar{x}_{k}\right),\\
	& \theta_2=-\alpha_{k}\left(\mathbf{I}_{nd}-\frac{\mathbf{1}u^\intercal}{n}\otimes\mathbf{I}_{d}\right)y_k,
	\end{align*}
	we have $x_{k+1}-\mathbf{1}\otimes\bar{x}_{k+1}=\theta_1+\theta_2$ 
	and
	\begin{equation}\label{x-bar}
	\begin{aligned}
	&\mathbb{E}\left[\|x_{k+1}-\mathbf{1}\otimes\bar{x}_{k+1}\|_\A^2\right]\\
	&\le (1+\tau)\mathbb{E}\left[\left\|\left(\tilde{\mathbf{A}}-\frac{\mathbf{1}u^\intercal}{n}\otimes\mathbf{I}_{d}\right)\left(x_k-\mathbf{1}\otimes\bar{x}_{k}\right)\right\|_\A^2\right]\\
	&\quad+\left(1+\frac{1}{\tau}\right)\mathbb{E}\left[\left\|\alpha_{k}\left(\mathbf{I}_{nd}-\frac{\mathbf{1}u^\intercal}{n}\otimes\mathbf{I}_{d}\right)y_k\right\|_\A^2\right]\\
	&\le (1+\tau)\tau_\A^2\mathbb{E}\left[\left\|x_k-\mathbf{1}\otimes\bar{x}_{k}\right\|_\A^2\right]\\
	&\quad+\left(1+\frac{1}{\tau}\right)\left\|\mathbf{I}_{n}-\frac{\mathbf{1}u^\intercal}{n}\right\|_\A^2\X^2\mathbb{E}\left[\left\|y_k\right\|^2\right]\\
	&\le\frac{1+\tau_\mathbf{A}^2}{2}\mathbb{E}\left[\left\|x_k-\mathbf{1}\otimes\bar{x}_{k}\right\|_\A^2\right]\\
	&\quad+\alpha_{k}^2\frac{1+\tau_\mathbf{A}^2}{1-\tau_\mathbf{A}^2}\left\|\mathbf{I}_{n}-\frac{\mathbf{1}u^\intercal}{n}\right\|_\A^2\X^2\mathbb{E}\left[\left\|y_k\right\|^2\right],
	\end{aligned}
	\end{equation}
	where $\tau_\A$ is defined in (\ref{consensus-para}), the last inequality follows from the setting $\tau=(1-\tau_\A^2)/(2\tau_\A^2)$. For the term $\mathbb{E}\left[\left\|y_k\right\|^2\right]$,
	{\small	\begin{equation}\label{y-error}
		\begin{aligned}
		\mathbb{E}\left[\left\|y_k\right\|^2\right]
		&=\mathbb{E}\left[\left\|y_k^{'}-v\otimes\bar{y}^{'}_k+v\otimes\left(\bar{y}^{'}_k-\frac{1}{n}\nabla f(\bar{x}_k)\right)\right.\right.\\
		&\quad\left.\left.+\frac{v}{n}\otimes \nabla f(\bar{x}_k)+\xi_k\right\|^2\right]\\
		&\le 4 \mathbb{E}\left[\left\|y_k^{'}-v\otimes\bar{y}^{'}_k\right\|^2+\left\|v\otimes\left(\bar{y}^{'}_k-\frac{1}{n}\nabla f(\bar{x}_k)\right)\right\|^2\right.\\
		&\quad\left.+\left\|\frac{v}{n}\otimes \nabla f(\bar{x}_k)\right\|^2+\left\|\xi_k\right\|^2\right]\\
		&\le 4 \mathbb{E}\left[\X^2\left\|y_k^{'}-v\otimes\bar{y}^{'}_k\right\|_\Z^2+\frac{(\|v\|L\X)^2}{n}\left\|x_k-\mathbf{1}\otimes\bar{x}_{k}\right\|_\A^2\right.\\
		&\quad\left.+\|v\|^2L^2\left\|\bar{x}_k-x^*\right\|^2\right]+4\frac{c_b^2}{(1-\tau_\Z)^2}nU_1,
		\end{aligned}
		\end{equation}	}
	where the second inequality follows from (\ref{ie-4}), (\ref{ie-5}) and (\ref{ie-6}). Substitute (\ref{y-error}) into (\ref{x-bar}),
	\begin{equation}\label{ine-1}
	\begin{aligned}
	&\mathbb{E}\left[\|x_{k+1}-\mathbf{1}\otimes\bar{x}_{k+1}\|_\A^2\right]\\
	&\le\frac{1+\tau_\mathbf{A}^2}{2}\mathbb{E}\left[\left\|x_k-\mathbf{1}\otimes \bar{x}_{k}\right\|_\A^2\right]+c_1\alpha_k^2\frac{c_b^2}{(1-\tau_\Z)^2}nU_1\\
	&+c_1\alpha_{k}^2\mathbb{E}\left[\X^2\left\|y_k^{'}-v\otimes\bar{y}^{'}_k\right\|_\Z^2+\frac{\|v\|^2L^2}{n}\X^2\left\|x_k-\mathbf{1}\otimes \bar{x}_{k}\right\|_\A^2\right.\\
	&\quad\left.+\|v\|^2L^2\left\|\bar{x}_k-x^*\right\|^2\right],
	\end{aligned}
	\end{equation}	
	where constant $c_1=4\frac{1+\tau_\mathbf{A}^2}{1-\tau_\mathbf{A}^2}\left\|\mathbf{I}_{n}-\frac{\mathbf{1}u^\intercal}{n}\right\|_\A^2\X^2$.

	Choosing
	\begin{align*}
	&\theta_1=\left(\tilde{\mathbf{B}}-\frac{v\mathbf{1}^\intercal}{n}\otimes\mathbf{I}_{d}\right)\left(y_{k}^{'}-v\otimes \bar{y}_{k}^{'}\right),\\
	&\theta_2=\left(\mathbf{I}_{nd}-\frac{v\mathbf{1}^\intercal}{n}\otimes\mathbf{I}_{d}\right)\left(\nabla F_{k+1}-\nabla F_k\right)
	\end{align*}
	in (\ref{general-form}), by the definitions of $y_{k+1}^{'}$ and $\bar{y}_{k+1}^{'}$, we have $y_{k+1}^{'}-\mathbf{1}\otimes\bar{y}_{k+1}^{'}=\theta_1+\theta_2$ and
	\begin{equation}\label{consensus-3}
	\begin{aligned}
	&\mathbb{E}\left[\|y_{k+1}^{'}-v\otimes\bar{y}_{k+1}^{'}\|_\Z^2\right]\\
	&\le(1+\tau)\mathbb{E}\left[\left\|\left(\tilde{\mathbf{B}}-\frac{v\mathbf{1}^\intercal}{n}\otimes\mathbf{I}_{d}\right)\left(y_{k}^{'}-v\otimes\bar{y}_{k}^{'}\right)\right\|_\Z^2\right]\\
	&\quad+\left(1+\frac{1}{\tau}\right)\mathbb{E}\left[\left\|\left(\mathbf{I}_{nd}-\frac{v\mathbf{1}^\intercal}{n}\otimes\mathbf{I}_{d}\right)\left(\nabla F_{k+1}-\nabla F_k\right)\right\|_\Z^2\right]\\
	&\le
	(1+\tau)\tau_\Z^2\mathbb{E}\left[\left\|y_{k}^{'}-v\otimes \bar{y}_{k}^{'}\right\|_\Z^2\right]\\
	&\quad+\left(1+\frac{1}{\tau}\right)\left\|\mathbf{I}_{nd}-\frac{v\mathbf{1}^\intercal}{n}\otimes\mathbf{I}_{d}\right\|_\Z^2\X^2\mathbb{E}\left[\left\|\nabla F_{k+1}-\nabla F_k\right\|^2\right]\\
	&\le
	\frac{1+\tau_\mathbf{B}^2}{2}\mathbb{E}\left[\left\|y_{k}^{'}-v\otimes \bar{y}_{k}^{'}\right\|_\Z^2\right]\\
	&+\frac{1+\tau_\mathbf{B}^2}{1-\tau_\mathbf{B}^2}\left\|\mathbf{I}_{nd}-\frac{v\mathbf{1}^\intercal}{n}\otimes\mathbf{I}_{d}\right\|_\Z^2\X^2\mathbb{E}\left[\left\|\nabla F_{k+1}-\nabla F_k\right\|^2\right],
	\end{aligned}
	\end{equation}
	where $\tau_\Z$ is defined in (\ref{consensus-para-1}), the last inequality follows from the setting $\tau=(1-\tau_\Z^2)/(2\tau_\Z^2)$. Moreover,
	\begin{equation}\label{F-error}
	\begin{aligned}
	&\mathbb{E}\left[\left\|\nabla F_{k+1}-\nabla F_k\right\|^2\right]\\
	&\le L^2\mathbb{E}\left[\left\|x_{k+1}-x_{k}\right\|^2\right]\\
	&=L^2 \mathbb{E}\left[\left\|\left(\tilde{\mathbf{A}}-\mathbf{I}_{nd}\right)\left(x_k-\mathbf{1}\otimes \bar{x}_k\right)-\alpha_ky_k\right\|^2\right]\\
	&\le 2L^2 \X^2\mathbb{E}\left[\left\|\tilde{\mathbf{A}}-\mathbf{I}_{nd}\right\|^2 \left\|x_k-\mathbf{1}\otimes \bar{x}_k\right\|_\A^2\right]+2L^2\alpha_k^2\mathbb{E}\left[\left\|y_k\right\|^2\right],
	\end{aligned}
	\end{equation}	
	where the first inequality follows from the Lipschitz continuity of $\nabla f_i(\cdot)$ as shown by (\ref{ie-1}), the equality follows from the fact $\left(\tilde{\mathbf{A}}-\mathbf{I}_{nd}\right)(\mathbf{1}\otimes \bar{x}_k)=\mathbf{0}$ by the row stochasticity of $\mathbf{A}$. Substitute (\ref{y-error}) and (\ref{F-error}) into (\ref{consensus-3}),
	{\small\begin{equation}\label{ine-2}
		\begin{aligned}
		&\mathbb{E}\left[\|y_{k+1}^{'}-v\otimes\bar{y}_{k+1}^{'}\|_\Z^2\right]\\
		&\le \frac{1+\tau_\mathbf{B}^2}{2}\mathbb{E}\left[\left\|y_{k}^{'}-v\otimes \bar{y}_{k}^{'}\right\|_\Z^2\right]\\
		&\quad+c_2\left\|\tilde{\mathbf{A}}-\mathbf{I}_{nd}\right\|^2\X^2\mathbb{E}\left[\left\|x_k-\mathbf{1}\otimes \bar{x}_k\right\|_\A^2\right]+\alpha_k^2\frac{4c_2c_b^2nU_1}{(1-\tau_\Z)^2}\\
		&\quad+4c_2\alpha_{k}^2\mathbb{E}\left[\X^2\left\|y_k^{'}-v\otimes\bar{y}^{'}_k\right\|_\Z^2+\frac{\|v\|^2L^2}{n}\X^2\left\|x_k-\mathbf{1}\otimes \bar{x}_{k}\right\|_\A^2\right.\\
		&\quad\left.+\|v\|^2L^2\left\|\bar{x}_k-x^*\right\|^2\right],
		\end{aligned}
		\end{equation}}
	where the constant
	\begin{equation}\label{parameter-2}
	c_2=2L^2\frac{1+\tau_\mathbf{B}^2}{1-\tau_\mathbf{B}^2}\left\|\mathbf{I}_{nd}-\frac{v\mathbf{1}^\intercal}{n}\otimes\mathbf{I}_{d}\right\|_\Z^2\X^2.
	\end{equation}

	Multiplying $c_3=\frac{1-\tau_\mathbf{A}^2}{4c_2\left\|\tilde{\mathbf{A}}-\mathbf{I}_{nd}\right\|^2\X^2}$ on both sides of inequality (\ref{ine-2}),
	{\small\begin{equation*}
		\begin{aligned}
		&c_3\mathbb{E}\left[\|y_{k+1}^{'}-v\otimes\bar{y}_{k+1}^{'}\|_\Z^2\right]\\
		&\le \frac{1+\tau_\mathbf{B}^2}{2}c_3\mathbb{E}\left[\left\|y_{k}^{'}-v\otimes \bar{y}_{k}^{'}\right\|_\Z^2\right]+\frac{1-\tau_\mathbf{A}^2}{4}\mathbb{E}\left[\left\|x_k-\mathbf{1}\otimes \bar{x}_k\right\|_\A^2\right]\\
		&\quad+4c_3c_2\alpha_k^2\frac{c_b^2}{(1-\tau_\Z)^2}nU_1+4c_3c_2\alpha_{k}^2\mathbb{E}\left[\X^2\left\|y_k^{'}-v\otimes\bar{y}^{'}_k\right\|_\Z^2\right.\\
		&\quad\left.+\frac{\|v\|^2L^2}{n}\X^2\left\|x_k-\mathbf{1}\otimes \bar{x}_{k}\right\|_\A^2+\|v\|^2L^2\left\|\bar{x}_k-x^*\right\|^2\right].
		\end{aligned}
		\end{equation*}}
	Substitute above inequality into (\ref{ine-1}), we have
	{\small\begin{equation}\label{ine-3}
		\begin{aligned}
		&\mathbb{E}\left[\|x_{k+1}-\mathbf{1}\otimes\bar{x}_{k+1}\|_\A^2\right]+c_3\mathbb{E}\left[\|y_{k+1}^{'}-v\bar{y}_{k+1}^{'}\|_\Z^2\right]\\
		&\le \frac{1+\tau_\mathbf{B}^2}{2}c_3\mathbb{E}\left[\left\|y_{k}^{'}-v\otimes \bar{y}_{k}^{'}\right\|_\Z^2\right]+\frac{3+\tau_\mathbf{A}^2}{4}\mathbb{E}\left[\left\|x_k-\mathbf{1}\otimes \bar{x}_k\right\|_\A^2\right]\\
		&\quad+\frac{(4c_3c_2+c_1)c_b^2nU_1}{(1-\tau_\Z)^2}\alpha_k^2+(4c_3c_2+c_1)\alpha_k^2\mathbb{E}\left[\X^2\left\|y_k^{'}-v\otimes\bar{y}^{'}_k\right\|_\Z^2\right.\\
		&\quad\left.+\frac{\|v\|^2L^2}{n}\X^2\left\|x_k-\mathbf{1}\otimes \bar{x}_{k}\right\|_\A^2+\|v\|^2L^2\left\|\bar{x}_k-x^*\right\|^2\right].
		\end{aligned}
		\end{equation}}
	Denote
	\begin{align*}
	&	u_k=\mathbb{E}\left[\|x_{k}-\mathbf{1}\otimes\bar{x}_{k}\|^2\right]+c_3\mathbb{E}\left[\|y_{k}^{'}-v\otimes\bar{y}_{k}^{'}\|^2\right],\\
	& v_k=\mathbb{E}\left[\|\bar{x}_{k}-x^*\|^2\right],\\
	& \rho_k=\max\left\{\frac{1+\tau_\mathbf{B}^2}{2},~\frac{3+\tau_\mathbf{A}^2}{4}\right\},
	\end{align*}
	$\gamma_{k}=\alpha_{k}$ and $\rho=\tau_\Z$. Then by inequalities (\ref{ine-0}) and (\ref{ine-3}),
	\begin{align}
	\label{form-3}u_{k+1}&\le\rho_ku_k+M\gamma_{k}\sqrt{u_k(1+u_k+v_k)}\\
	&\quad+M\gamma_{k}\left(\gamma_{k}+\sum_{t=0}^k\gamma_{t}\rho^{k-t}u_t+\sum_{t=0}^k\gamma_{t}\rho^{k-t}v_t\right),\notag\\
	\label{form-4}v_{k+1}&\le v_k+Mu_k+M\gamma_{k}\sqrt{u_k(1+u_k+v_k)}\\
	&+M\gamma_{k}\left(\gamma_{k}+\sum_{t=0}^k\gamma_{t}\rho^{k-t}u_t+\sum_{t=0}^k\gamma_{t}\rho^{k-t}v_t\right),\notag
	\end{align}
	where
	\begin{align*}
	&M=\max\left\{2\sqrt{2}\frac{u^\intercal vL\X}{n^{1.5}},~2\sqrt{2}\frac{\|u\|\X}{n\sqrt{c_3}},~(4c_3c_2+c_1)\|v\|^2L^2\bar{c}^2,\right.\\&\quad\quad\quad\quad\quad\left.c_0,~(4c_3c_2+c_1)\X^2,~\frac{(4c_3c_2+c_1)c_b^2nU_1}{(1-\tau_\Z)^2}\right\}.
	\end{align*}
	Obviously, (\ref{form-3}) and (\ref{form-4}) fall into the forms of (\ref{form-1}) and (\ref{form-2}).
	
	\textbf{Step 2.}
	By the facts  $\rho_k=\max\left\{\frac{1+\tau_\mathbf{B}^2}{2},~\frac{3+\tau_\mathbf{A}^2}{4}\right\}$ and $\gamma_k=\alpha_{k}$, there exists a positive integer $k_0$ such that $\gamma_{k}$,  $\rho_{k}$ are $[0,1]$-valued when $k\ge k_0$ (without loss of generality, suppose $k_0=0$). Then condition (i) of Lemma \ref{lem:ieq-1} holds.
	
	Let  $\phi_k=1/\alpha_{k}^2$. Obviously,
	\begin{align*}
	&\limsup_{k\rightarrow\infty}\left(\gamma_{k}\sqrt{\phi_k}+\frac{\phi_{k-1}}{\phi_k}\right)<\infty,\quad \sum_{k=0}^\infty\phi_k^{-1}<\infty,\\
	&\liminf_{k\rightarrow\infty}\left(\gamma_{k}\sqrt{\phi_k}\right)^{-1}\left(\frac{\phi_{k-1}}{\phi_k}-\rho_k\right)>0,
	\end{align*}
	which implies the condition (ii) of Lemma \ref{lem:ieq-1}. Then by Lemma \ref{lem:ieq-1},   $\sup_{k}\mathbb{E}\left[\|\bar{x}_{k+1}-x^*\|^2\right]<\infty$ and
	\begin{equation*}
	\sup_{k} \frac{1}{\alpha_{k}^2}\left(\mathbb{E}\left[\|x_{k}-\mathbf{1}\otimes\bar{x}_{k}\|^2\right]+c_3\mathbb{E}\left[\|y_{k}^{'}-\mathbf{1}\otimes\bar{y}_{k}^{'}\|^2\right]\right)<\infty.
	\end{equation*}
	The proof is complete.
\end{proof}

The proof of Lemma \ref{lem:rate} follows the steps in the proof of \cite[Lemma 5]{bianchi2013performance}, where the stability and agreement of distributed stochastic gradient descent algorithm are obtained. In \cite{bianchi2013performance}, the weight matrices are double stochastic and the stochastic gradient noise forms the martingale difference sequence, which ensures the technical result \cite[Lemma 3]{bianchi2013performance} can be applied. For $\mathcal{S}$-$\mathcal{AB}$, the weight matrices are not doubly stochastic and the stochastic noise accumulated during gradient tracking steps forms the autoregressive moving average processes. We need to present an extension of \cite[Lemma 3]{bianchi2013performance}, which is suitable for $\mathcal{S}$-$\mathcal{AB}$ first. Then we analyze the intrinsic structure of the accumulated stochastic noise and arrive at the coupled relationship between optimality gap and agreement errors where the extended result can be applied.
\begin{thm}\label{thm:rate}
Suppose that Assumptions \ref{ass:objective} and \ref{ass:matrix} hold. Let step-size $\alpha_{k}=a/(k+b)^\alpha$, where $\alpha\in (1/2,1)$, positive scalars $a,b$ satisfy  $\frac{a}{b^\alpha}\le \min\{1,\frac{n}{u^\intercal vL}\}$, $L\define\max_{1\leq i \leq n}\sqrt{\mathbb{E}[L_{i}^2(\zeta_i)]}$. Then
\begin{equation}\label{rate-1}
\mathbb{E}\left[\|\bar{x}_k-x^*\|^2\right]=\mathcal{O}\left(\alpha_k\right).
\end{equation}
Moreover, if $\alpha_{k}=a/(k+b)$ and positive scalars $a,b$ satisfy  $\frac{n^2}{u^\intercal v\mu }<a\le \min\{1,\frac{n}{u^\intercal vL}\}b$, 
$$\mathbb{E}\left[\|\bar{x}_k-x^*\|^2\right]=\mathcal{O}\left(\frac{1}{k}\right).$$
\end{thm}
\begin{proof}
Combine the inequalities (\ref{eq-auxi-3}) with (\ref{term3}),
{\footnotesize\begin{equation}\label{ieq-0}
		\begin{aligned}
		&\mathbb{E}\left[\|\bar{x}_{k+1}-x^*\|^2\right]\\
		&\le\left(1-\frac{u^\intercal v\mu\alpha_{k}}{n^2}\right)^2\mathbb{E}\left[\left\|\bar{x}_{k}-x^*\right\|^2\right]\\
		&\quad+3\alpha_{k}^2\mathbb{E}\left[\left(\frac{u^\intercal v}{n}\right)^2\frac{L^2\X^2}{n}\left\|x_k-\mathbf{1}\otimes\bar{x}_k\right\|_\mathbf{A}^2\right]\\
		&\quad+3\alpha_{k}^2\mathbb{E}\left[\frac{\|u\|^2\X^2}{n^2}\left\|y_k^{'}-v\otimes\bar{y}^{'}_k\right\|_\mathbf{B}^2\right]+3\alpha_{k}^2\frac{\|u\|^2c_b^2U_1}{n(1-\tau_\Z)^2}\\
		&\quad-2\mathbb{E}\left[\left\langle \bar{x}_{k}-x^*, \alpha_{k}\left(\frac{u^\intercal}{n}\otimes\mathbf{I}_{d}\right)\left(y_k^{'}-\frac{v}{n}\otimes \nabla f(\bar{x}_k)+\xi_k\right)\right\rangle\right]\\
		&\quad+\frac{u^\intercal v\alpha_{k}^2}{n^2}\mathbb{E}\left[3n^2L^2\|\bar{x}_k-x^*\|^2+\left(\frac{u^\intercal v}{n}\right)^2\frac{L^2\X^2}{n}\left\|x_k-\mathbf{1}\otimes\bar{x}_k\right\|_\mathbf{A}^2\right.\\
		&\quad\left.+\frac{\|u\|^2\X^2}{n^2}\left\|y_k^{'}-v\otimes\bar{y}^{'}_k\right\|_\mathbf{B}^2\right]+\frac{u^\intercal v\|u\|^2c_b^2U_1}{n^3(1-\tau_\Z)^2}\alpha_{k}^2,
		\end{aligned}
		\end{equation}}
	where  $\bar{y}^{'}_k\define1/n\sum_{i=1}^ny_{i,k}^{'},~c_b=\max\left\{\X^2,\frac{\left\|\mathbf{B}-\mathbf{I}_{n}\right\|}{\tau_\Z}\X^2\right\}$, $\X$ is defined in (\ref{norm-bound-1}). For the fifth term on the right hand side of (\ref{ieq-0}),
	{\footnotesize\begin{equation}\label{ieq-1}
		\begin{aligned}
		&-2\mathbb{E}\left[\left\langle \bar{x}_{k}-x^*, \alpha_{k}\left(\frac{u^\intercal}{n}\otimes\mathbf{I}_{d}\right)\left(y_k^{'}-\frac{v}{n}\otimes \nabla f(\bar{x}_k)+\xi_k\right)\right\rangle\right]\\
		&\le \frac{\alpha_{k}\|u\|}{n}\mathbb{E}\left[\tau\left\| \bar{x}_{k}-x^*\right\|^2\right]-2\mathbb{E}\left[\left\langle \bar{x}_{k}-x^*, \alpha_{k}\left(\frac{u^\intercal}{n}\otimes\mathbf{I}_{d}\right)\xi_k\right\rangle\right]\\
		&\quad+\frac{\alpha_{k}\|u\|}{n}\mathbb{E}\left[\frac{1}{\tau}\left\|y_k^{'}-v\otimes \bar{y}^{'}_k+v\otimes \left(\bar{y}^{'}_k- \frac{1}{n}\nabla f(\bar{x}_k)\right)\right\|^2\right]\\
		&\le \frac{\tau \|u\|\alpha_k}{n}\mathbb{E}\left[\left\|\bar{x}_{k}-x^*\right\|^2\right]+\frac{2\|u\|\alpha_k}{n\tau}\mathbb{E}\left[\left\|y_k^{'}-v\otimes \bar{y}^{'}_k\right\|^2\right]\\
		&\quad+\frac{2\|u\|\|v\|^2L^2\X^2\alpha_k}{n^2\tau}\mathbb{E}\left[\left\|x_k-\mathbf{1}\otimes\bar{x}_k\right\|_\mathbf{A}^2\right]\\
		&\quad+ \alpha_{k}\frac{2\|u\|^2\X}{n^2(1-\tau_\Z)}\sum_{t=0}^{k-1}\alpha_{t}\tau_\Z^{k-t}\left(\mathbb{E}\left[\X^2\|y_t^{'}-v\otimes \bar{y}^{'}_t\|_\mathbf{B}^2\right]\right.\\
		&\quad\left.+\frac{\|v\|^2L^2\X^2}{n}\mathbb{E}\left[\|x_t-\mathbf{1}\otimes\bar{x}_t\|_\mathbf{A}^2\right]+\frac{\|v\|^2L^2}{n^2}\mathbb{E}\left[\|\bar{x}_t-x^*\|^2\right]\right)\\
		&\quad+\alpha_{k}^2\frac{2.5\|u\|^2\X c_b^2U_1 c_\rho}{n(1-\tau_\Z)^3},
		\end{aligned}
		\end{equation}}
	where $\tau$ could be any positive number, the inequality follows from (\ref{ie-4}) in the proof of Lemma \ref{lem:ieq-1} and (\ref{2-momon-2}) in Appendix \ref{apd:proof of Lem-rate}. Substitute (\ref{ieq-1}) into (\ref{ieq-0}), then by Lemma \ref{lem:rate} and Lemma \ref{lem:weighted-seq},
	{\small\begin{equation*}
		\begin{aligned}
		&\mathbb{E}\left[\|\bar{x}_{k+1}-x^*\|^2\right]\\
		&\le\left[\left(1-\frac{u^\intercal v\mu\alpha_{k}}{n^2}\right)^2+\frac{\tau \|u\|\alpha_k}{n} \right]\mathbb{E}\left[\left\|\bar{x}_{k}-x^*\right\|^2\right]\\
		&\quad+\mathcal{O}\left(\frac{\alpha_k^3}{\tau}\right)+\mathcal{O}\left(\alpha_k^2\right)\\
		&\le\left(1-\frac{u^\intercal v\mu\alpha_{k}}{n^2}+\left(\frac{u^\intercal v\mu\alpha_{k}}{n^2}\right)^2\right)\mathbb{E}\left[\left\|\bar{x}_{k}-x^*\right\|^2\right]+\mathcal{O}\left(\alpha_k^2\right)\\
		&\le \left(1-\frac{u^\intercal v\mu\alpha_{k}}{n^2}\right)\mathbb{E}\left[\left\|\bar{x}_{k}-x^*\right\|^2\right]+\left(\frac{u^\intercal v\mu}{n^2}\right)^2U_2\alpha_{k}^2+\mathcal{O}\left(\alpha_k^2\right),
		\end{aligned}
		\end{equation*}}
	where the equality follows from the setting $\tau=\left(\frac{u^\intercal v\mu}{n^2}\right)/\left(\frac{\|u\|}{n}\right)$. Then by  \cite[Lemma 5 in Chapter 2]{polyak1987Introduction}, we have 
	\begin{equation*}
	\mathbb{E}\left[\|\bar{x}_{k+1}-x^*\|^2\right]= \mathcal{O}\left(\alpha_k\right).
	\end{equation*}
	Note that $\lim_{k\rightarrow\infty} \frac{\alpha_{k-1}}{\alpha_{k}}=1$, then (\ref{rate-1}) holds.
	If $\alpha_{k}=a/(k+b)$ and positive scalars $a,b$ satisfy that $\frac{n^2}{u^\intercal v\mu }<a\le \min\{1,\frac{2n}{u^\intercal vL}\}b$, we have $\frac{u^\intercal v\mu a}{n^2}>1$ and then \cite[Lemma 4 in Chapter 2]{polyak1987Introduction} implies
	$$\mathbb{E}\left[\|\bar{x}_k-x^*\|^2\right]=\mathcal{O}\left(\frac{1}{k}\right).$$
	The proof is complete.	
\end{proof}

Combining  with Lemma \ref{lem:rate}, Theorem \ref{thm:rate}  implies that $\mathbb{E}\left[\|x_{i,k}-x^*\|^2\right]=\mathcal{O}\left(\frac{1}{k}\right)$ for any $i\in \mathcal{V}$, which is the optimal rate of  stochastic gradient method \cite{Rakhlin2012making}. 
Complement to \cite[Theorem 2]{pu2020distributed}, Theorem \ref{thm:rate} focuses on the undoubly stochastic weight matrices and directed graphs. 
%

In what follows, we show the asymptotic normality of $\mathcal{S}$-$\mathcal{AB}$ algorithm.
\begin{thm}\label{thm:asym norm}
	Suppose that Assumptions \ref{ass:objective} and \ref{ass:matrix} hold. Let step-size $\alpha_{k}=a/(k+b)^\alpha$, where $\alpha\in (1/2,1)$, positive scalars $a,b$ satisfy  $\frac{a}{b^\alpha}\le \min\{1,\frac{n}{u^\intercal vL}\}$, $L\define\max_{1\leq i \leq n}\sqrt{\mathbb{E}[L_{i}^2(\zeta_i)]}$. Then for any $i\in\mathcal{V}$, $x_{i,k}$ satisfies
	\begin{equation}\label{limit distribution}
	\frac{1}{\sqrt{k}}\sum_{t=0}^{k-1}\left(x_{i,t}-x^*\right)\stackrel{d}{\longrightarrow} N\left(\mathbf{0},\mathbf{H}^{-1}\mathbf{S}\mathbf{H}^{-1}\right),
	\end{equation}
	where $\mathbf{H}\define \nabla^2 f(x^*)$ and  $\mathbf{S}\define\Cov(\sum_{j=1}^n\nabla g_j(x^*,\zeta_{j}))$.
\end{thm}
\begin{proof}
	By Lemma \ref{lem:rate},
	\begin{equation*}
	\begin{aligned}
	&\mathbb{E}\left[\left\|\frac{1}{\sqrt{k}}\sum_{t=0}^{k-1}\left(\bar{x}_{t}-x^*\right)-\frac{1}{\sqrt{k}}\sum_{t=0}^{k-1}\left(x_{i,t}-x^*\right)\right\|\right]\\
	&\le \frac{1}{\sqrt{k}}\sum_{t=0}^{k-1}\sqrt{\mathbb{E}\left[\|x_{t}-\mathbf{1}\otimes\bar{x}_{t}\|^2\right]}\le\frac{\sqrt{U_2}}{\sqrt{k}}\sum_{t=0}^{k-1}\alpha_t\rightarrow 0.
	\end{aligned}
	\end{equation*}
	Then  Slutsky's theorem implies (\ref{limit distribution}) if
	\begin{equation}\label{res-asym-norm}
	\frac{1}{\sqrt{k}}\sum_{t=0}^{k-1}\left(\bar{x}_{t}-x^*\right)\stackrel{d}{\rightarrow} N\left(\mathbf{0},\mathbf{H}^{-1}\mathbf{S}\mathbf{H}^{-1}\right)
	\end{equation}
	holds. In what follows, we show (\ref{res-asym-norm}) by Lemma \ref{lem:asym-norm} in Appendix \ref{apd:thm-asym-norm}.
	
	Firstly, we rewrite the recursion $\bar{x}_k-x^*$ in the form of (\ref{norm form}) in Lemma \ref{lem:asym-norm}. By the equality (\ref{consensus-new-form}),
	\begin{equation}\label{consensus-new-form-1}
	\begin{aligned}
	&\bar{x}_{k+1}-x^*\\
	&=\bar{x}_{k}-x^*-\tilde{\alpha}_k\frac{1}{n}\nabla f(\bar{x}_k)-\tilde{\alpha}_k\left(\bar{y}^{'}_k-\frac{1}{n}\nabla f(\bar{x}_k)\right)\\
	&\quad-\alpha_k\left(\frac{u^\intercal}{n}\otimes\mathbf{I}_{d}\right)\left(y_k^{'}-v\otimes\bar{y}^{'}_k\right)-\alpha_{k}\left(\frac{u^\intercal}{n}\otimes\mathbf{I}_{d}\right)\xi_k,
	\end{aligned}
	\end{equation}	
	where $\tilde{\alpha}_k\define\frac{u^\intercal v}{n}\alpha_k$.	Denote
	\begin{equation*}\label{sme def-1}
	\Delta_{k}=\bar{x}_k-x^*,\quad h(x)=-\frac{1}{n}\nabla f(x+x^*), \quad V(x)=\frac{1}{2}\|x\|^2
	\end{equation*}
	and
	\begin{align}
	\label{some def-2}&\eta_k=-\left(\bar{y}^{'}_k-\frac{1}{n}\nabla f(\bar{x}_k)\right)-\frac{n}{u^\intercal v}\left(\frac{u^\intercal}{n}\otimes\mathbf{I}_{d}\right)\left(y_k^{'}-v\otimes\bar{y}^{'}_k\right),\\
	&\mu_k=-\frac{n}{u^\intercal v}\left(\frac{u^\intercal}{n}\otimes\mathbf{I}_{d}\right)\xi_k.\notag
	\end{align}
	Then the linear recursion (\ref{consensus-new-form-1}) can be rewritten as
	\begin{equation}\label{consensus-linear}
	\Delta_{k+1}=\Delta_{k}+\tilde{\alpha}_kh(\Delta_{k})+\tilde{\alpha}_k\left(\eta_k+\mu_k\right),
	\end{equation}
	which is in the form of (\ref{norm form}) in Lemma \ref{lem:asym-norm}.
	
	Next, we verify the conditions (C0)-(C3) of Lemma \ref{lem:asym-norm}. By the definition of $\alpha_{k}$, (C0) holds. By the strong convexity of $f(x)$,
	\begin{align*}
	h(x)^\intercal \nabla V(x)&=\left\langle -\frac{1}{n}\nabla f(x+x^*), (x+x^*)-x^* \right\rangle\\
	&< \frac{1}{n}(f(x^*)-f(x+x^*))<0,\quad\forall x\neq 0,
	\end{align*}
	which implies (C1) of Lemma \ref{lem:asym-norm}.
	
	Set  $\mathbf{G}=-\frac{1}{n}\nabla^2f(x^*)$. Note that $\nabla^2f(x^*)$ is positive definite as  $f(x)$ is strongly convex and then $\mathbf{G}$ is stable. Subsequently, combining with Assumption \ref{ass:objective} (ii), (C2) of Lemma \ref{lem:asym-norm} holds.
	
	We are left to verify (C3) of Lemma \ref{lem:asym-norm}. We first verify (\ref{apd:cond-1}) of (C3). By the definition (\ref{auxi-seq}) of $\xi_{k}$,
	\begin{equation*}
	\xi_{k}=\tilde{\mathbf{B}}\xi_{k-1}+\epsilon_{k}-\epsilon_{k-1}=\sum_{t=0}^{k-1}\tilde{\mathbf{B}}^{k-1-t}\left(\tilde{\mathbf{B}}-\mathbf{I}_{nd}\right)\epsilon_{t}+\epsilon_{k},
	\end{equation*}
	and then
	\begin{equation*}
	\sum_{t=0}^{k-1}\mu_{t}=-\sum_{t=0}^{k-1}\frac{n}{u^\intercal v}\left(\frac{u^\intercal}{n}\otimes\mathbf{I}_{d}\right)\xi_t=-\sum_{t=0}^{k-1}\mathbf{D}(k-1,t)\epsilon_t,
	\end{equation*}
	where $\mathbf{D}(k-1,t)=\frac{n}{u^\intercal v}\left(\frac{u^\intercal}{n}\otimes\mathbf{I}_{d}\right)\tilde{\mathbf{B}}^{k-1-t}$ for $t<k-1$, $\mathbf{D}(k-1,k-1)=\frac{n}{u^\intercal v}\left(\frac{u^\intercal}{n}\otimes\mathbf{I}_{d}\right)$.
	Note that $\{\epsilon_k\}$ is a martingale difference sequence and by Assumption \ref{ass:objective} (iii),
	\begin{equation}\label{MD-bound}
	\sup_k \mathbb{E}\left[\|\epsilon_k\|^2|\mathcal{F}_{k}\right]\le nU_1.
	\end{equation}
	Then by \cite[Lemma B.6.1, Appendix B.6]{chen2006stochastic}, $$\sum_{t=0}^{k-1}\left(\frac{\mathbf{1}^\intercal}{n}\otimes \mathbf{I}_d\right)\epsilon_t=\mathcal{O}(\sqrt{k}),\quad \text{a.s.}$$ which implies
	\begin{equation}\label{as-conv-0}
	\tilde{\alpha}_{k-1}\sum_{t=0}^{k-1}\left(\frac{\mathbf{1}^\intercal}{n}\otimes \mathbf{I}_d\right)\epsilon_t=\mathcal{O}(\sqrt{k}\tilde{\al}_{k-1})=\mathcal{O}(k^{1/2-\al})\quad \text{a.s.}
	\end{equation}
	Note also that
	{\small\begin{equation}\label{noi-consensus-0}
	\begin{aligned}
	&\mathbb{E}\left[\left\|\tilde{\alpha}_{k-1}\sum_{t=0}^{k-1}\mathbf{D}_{k-1,t}\epsilon_t-\tilde{\alpha}_{k-1}\sum_{t=0}^{k-1}\left(\frac{\mathbf{1}^\intercal}{n}\otimes \mathbf{I}_d\right)\epsilon_t\right\|^2\right]\\
	&=\tilde{\alpha}_{k-1}^2\sum_{t=0}^{k-1}\mathbb{E}\left[\left\|\left[\mathbf{D}_{k-1,t}-\left(\frac{\mathbf{1}^\intercal}{n}\otimes \mathbf{I}_d\right)\right]\epsilon_t\right\|^2\right]\\
	&\le \tilde{\alpha}_{k-1}^2\sum_{t=0}^{k-1}\left\|\mathbf{D}_{k-1,t}-\left(\frac{\mathbf{1}^\intercal}{n}\otimes \mathbf{I}_d\right)\right\|^2\mathbb{E}\left[\left\|\epsilon_t\right\|^2\right]\\
	&= \tilde{\alpha}_{k-1}^2\sum_{t=0}^{k-1}\left\|\frac{n}{u^\intercal v}\left(\frac{u^\intercal}{n}\otimes\mathbf{I}_{d}\right)\left(\tilde{\mathbf{B}}^{k-1-t}-\frac{v\mathbf{1}^\intercal}{n}\otimes \mathbf{I}_d\right)\right\|^2\mathbb{E}\left[\left\|\epsilon_t\right\|^2\right]\\
	&\le  \tilde{\alpha}_{k-1}^2\sum_{t=0}^{k-1}\left(\frac{\|u\|}{u^\intercal v}\right)^2 \left\|\tilde{\mathbf{B}}^{k-1-t}-\frac{v\mathbf{1}^\intercal}{n}\otimes \mathbf{I}_d\right\|^2nU_1,
	\end{aligned}
	\end{equation}}
	where $\tau_\Z<1$ is defined in (\ref{consensus-para-1}), the first equality follows from that $\mathbb{E}\left[\epsilon_{t_1}^\intercal\epsilon_{t_2}\right]=\mathbb{E}\left[\mathbb{E}\left[\epsilon_{t_1}^\intercal\epsilon_{t_2}|\mathcal{F}_{t_1+1}\right]\right]=0$ for any $t_1<t_2$, $	\mathcal{F}_0 =\sigma\{x_{i,0},i\in \mathcal{V}\},$
	\begin{equation*}
	\mathcal{F}_k=\sigma\{x_{i,0},\epsilon_{i,t}:i\in \mathcal{V}, 0\leq t\leq k-1\}~(k>0),
	\end{equation*}
	the second inequality follows by taking full expectation on both sides of (\ref{MD-bound}).
	By the column stochasticity of matrix $\mathbf{B}$,
	{\footnotesize\begin{align*}
	\left\|\tilde{\mathbf{B}}^{t}-\frac{v\mathbf{1}^\intercal}{n}\otimes \mathbf{I}_d\right\|^2&=\left\|\left(\tilde{\mathbf{B}}^{t-1}-\frac{v\mathbf{1}^\intercal}{n}\otimes \mathbf{I}_d\right)\left(\tilde{\mathbf{B}}-\frac{v\mathbf{1}^\intercal}{n}\otimes \mathbf{I}_d\right)\right\|^2\notag\\
	&\le \X^2\left\|\left(\tilde{\mathbf{B}}^{t-1}-\frac{v\mathbf{1}^\intercal}{n}\otimes \mathbf{I}_d\right)\left(\tilde{\mathbf{B}}-\frac{v\mathbf{1}^\intercal}{n}\otimes \mathbf{I}_d\right)\right\|_\mathbf{B}^2\notag\\
	&\le \X^2\tau_\Z\left\|\tilde{\mathbf{B}}^{t-1}-\frac{v\mathbf{1}^\intercal}{n}\otimes \mathbf{I}_d\right\|_\mathbf{B}^2\notag\\
	&\le \cdots\le \X^2\tau_\Z^t.
	\end{align*}}
	Substituting above inequality into (\ref{noi-consensus-0}),
	\begin{equation}\label{noi-consensus}
	\begin{aligned}
	&\mathbb{E}\left[\left\|\tilde{\alpha}_{k-1}\sum_{t=0}^{k-1}\mathbf{D}_{k-1,t}\epsilon_t-\tilde{\alpha}_{k-1}\sum_{t=0}^{k-1}\left(\frac{\mathbf{1}^\intercal}{n}\otimes \mathbf{I}_d\right)\epsilon_t\right\|^2\right]\\
	&\le  \tilde{\alpha}_{k-1}^2\sum_{t=0}^{k-1}\left(\frac{\|u\|\X}{u^\intercal v}\right)^2 \tau_\Z^{2(k-1-t)}nU_1\\
	&\le \left(\frac{\|u\|\X}{u^\intercal v}\right)^2\frac{n}{1-\tau_\Z^2}U_1\tilde{\alpha}_{k-1}^2,
	\end{aligned}
	\end{equation}
	and then
	\begin{equation}\label{MDS -sum}
	\begin{aligned}
	&\sum_{k=1}^\infty\mathbb{E}\left[\left\|\tilde{\alpha}_{k-1}\sum_{t=0}^{k-1}\mathbf{D}_{k-1,t}\epsilon_t-\tilde{\alpha}_{k-1}\sum_{t=0}^{k-1}\left(\frac{\mathbf{1}^\intercal}{n}\otimes \mathbf{I}_d\right)\epsilon_t\right\|^2\right]\\
	&\le \sum_{k=0}^\infty \left(\frac{\|u\|\X}{u^\intercal v}\right)^2\frac{1}{1-\tau_\Z^2}U_1\tilde{\alpha}_{k-1}^2<\infty.
	\end{aligned}
	\end{equation}
	By (\ref{MDS -sum}) and the monotone convergence theorem,
	\begin{equation}\label{as-conv-1}
	\tilde{\alpha}_{k-1}\sum_{t=0}^{k-1}\mathbf{D}_{k-1,t}\epsilon_t-\tilde{\alpha}_{k-1}\sum_{t=0}^{k-1}\left(\frac{\mathbf{1}^\intercal}{n}\otimes \mathbf{I}_d\right)\epsilon_t\rightarrow 0\quad \text{a.s.}
	\end{equation}
	Combine (\ref{as-conv-0}) and (\ref{as-conv-1}),
	\begin{equation}\label{veri-C2-1-1}
	\begin{aligned}
	&\tilde{\alpha}_{k-1}\sum_{t=0}^{k-1}\mu_{t}\\
	&=\left(\tilde{\alpha}_{k-1}\sum_{t=0}^{k-1}\mu_{t}+\tilde{\alpha}_{k-1}\sum_{t=0}^{k-1}\left(\frac{\mathbf{1}^\intercal}{n}\otimes \mathbf{I}_d\right)\epsilon_t\right)\\
	&\quad-\tilde{\alpha}_{k-1}\sum_{t=0}^{k-1}\left(\frac{\mathbf{1}^\intercal}{n}\otimes \mathbf{I}_d\right)\epsilon_t\\
	&=-\left(\tilde{\alpha}_{k-1}\sum_{t=0}^{k-1}\mathbf{D}_{k-1,t}\epsilon_t-\tilde{\alpha}_{k-1}\sum_{t=0}^{k-1}\left(\frac{\mathbf{1}^\intercal}{n}\otimes \mathbf{I}_d\right)\epsilon_t\right)\\
	&\quad-\tilde{\alpha}_{k-1}\sum_{t=0}^{k-1}\left(\frac{\mathbf{1}^\intercal}{n}\otimes \mathbf{I}_d\right)\epsilon_t
	\rightarrow 0 \quad \text{a.s.}
	\end{aligned}   
	\end{equation}
	By Lemma \ref{lem:noi-asym-norm} in Appendix \ref{apd:thm-asym-norm}, we have
	\begin{equation}\label{veri-C2-1-2}
	\frac{1}{\sqrt{k}}\sum_{t=0}^{k-1} \mu_t\stackrel{d}{\rightarrow} N\left(\mathbf{0},\frac{1}{n^2}\mathbf{S}\right).
	\end{equation}
	Subsequently, (\ref{veri-C2-1-1}) and (\ref{veri-C2-1-2}) verify (\ref{apd:cond-1}) of (C3).
	
	Note that
	\begin{equation}\label{mu-2-monm}
	\begin{aligned}
	\mathbb{E}\left[\|\mu_k\|^2\right]&=\mathbb{E}\left[\left\|\frac{n}{u^\intercal v}\left(\frac{u^\intercal}{n}\otimes\mathbf{I}_{d}\right)\xi_k\right\|^2\right]\\
	&\le \left(\frac{\|u\|}{u^\intercal v}\right)^2\mathbb{E}\left[\left\|\xi_k\right\|^2\right]\le \left(\frac{\|u\|}{u^\intercal v}\right)^2 \frac{c_b^2}{(1-\tau_\Z)^2}nU_1,
	\end{aligned}
	\end{equation}
	where $c_b=\max\left\{\X^2,\frac{\left\|\mathbf{B}-\mathbf{I}_{n}\right\|}{\tau_\Z}\X^2\right\}$, $\X$ is defined in (\ref{norm-bound-1}), the second inequality follows from the boundedness of $\mathbb{E}\left[\left\|\xi_k\right\|^2\right]$ by Lemma \ref{lem:noi-bound}(i) in Appendix \ref{apd:proof of Lem-rate}. Then $\sum_{s\ge -t}^{\infty}\left\|\mathbb{E}\left[ \mu_t\mu^\intercal_{s+t}\right]\right\|\le c$ (by Lemma \ref{lem:noi-asym-norm} in  Appendix \ref{apd:thm-asym-norm}) and (\ref{mu-2-monm}) imply (\ref{apd:cond-2}) of condition (C3).
	
	By the definition of $\eta_k$ in (\ref{some def-2}),
	\begin{equation*}
	\begin{aligned}
	\mathbb{E}\left[\|\eta_k\|^2\right]
	&\le \frac{2L^2}{n}\mathbb{E}\left[\|x_{k}-\mathbf{1}\otimes\bar{x}_{k}\|^2\right]\\
	&\quad+2\left(\frac{\|u\|}{u^\intercal v}\right)^2 \mathbb{E}\left[\left\|y_k^{'}-v\otimes\bar{y}_k^{'}
	\right\|^2\right]\\
	&\le\left(\frac{2L^2}{n}+2\left(\frac{\|u\|}{u^\intercal v}\right)^2\right)U_2\alpha_{k}^2,
	\end{aligned}
	\end{equation*}
	where the first inequality follows from the facts $\bar{y}^{'}_k=\frac{1}{n}\sum_{j=1}^n\nabla f_j(x_{i,k})$ and (\ref{ie-2}), the last inequality follows from Lemma \ref{lem:rate}. Then (\ref{apd:cond-3}) of (C3) holds.
	
	Summarizing above results, by Lemma \ref{lem:asym-norm} in Appendix \ref{apd:thm-asym-norm},
	\begin{equation*}
	\frac{1}{\sqrt{k}}\sum_{t=0}^{k-1}\left(\bar{x}_{k}-x^*\right)\stackrel{d}{\longrightarrow} N\left(\mathbf{0},\mathbf{H}^{-1}\mathbf{S}\mathbf{H}^{-1}\right).
	\end{equation*}
	The proof is complete.
\end{proof}

Theorem \ref{thm:asym norm}
shows the asymptotic normality  of
Polyak-Ruppert averaged  $\mathcal{S}$-$\mathcal{AB}$.
One of the applications of asymptotic normality is to do statistical inference to construct the confidence regions of the true solution. For doing so, we need    to estimate the covariance matrix $\mathbf{H}^{-1}\mathbf{S}\mathbf{H}^{-1}$  first.
Motivated by the online plug-in method in the seminal work \cite{chen2020statistical}, we propose the following distributed plug-in method.

\floatname{algorithm}{}
\renewcommand{\thealgorithm}{}
\begin{algorithm}[H]
	\caption{\textbf{Distributed plug-in method}.}\label{alg:dis plug-in-1}
	\textbf{Initialization:} For $i\in\mathcal{V}$, set vector $u_{i,0}=[0,\cdots,\underbrace{1}_{i-th},\cdots,0]^\intercal$, matrices $\mathbf{H}_{i,0}=\mathbf{S}_{i,0}=\mathbf{0}\in\mathbb{R}^{d\times d}$.\\
	\textbf{General step:} For any $k=1,2,\cdots$,
	\begin{itemize}
		\item[1:]Update $u_{i,k}$ by
		$u_{i,k}=\sum_{j=1}^na_{ij}u_{j,k-1}$.
		\item[2:]Update $\mathbf{S}_{i,k},~\mathbf{H}_{i,k}$ by
		{\small\begin{align}
			&\mathbf{S}_{i,k}=\dfrac{k}{k+1}\sum_{j=1}^n a_{ij} \mathbf{S}_{j,k-1}^{(1)}\notag\\
			\label{S}&\quad+\frac{\nabla g_{i,k}\left(\nabla g_{i,k}-\nabla g_{i,k-1}\right)^\intercal+\left(\nabla g_{i,k}-\nabla g_{i,k-1}\right)\nabla g_{i,k}^\intercal}{2(k+1)u_{i,k}(i)},\\
			\label{est-H}&\mathbf{H}_{i,k}=\dfrac{k}{k+1}\sum_{j=1}^n a_{ij} \mathbf{H}_{j,k-1}+\frac{\nabla^2 g_{i,k}}{(k+1)u_{i,k}(i)}
			\end{align}}
		respectively, where $u_{i,k}(i)$ is the i-th component of $u_{i,k}$, $\nabla g_{i,k}=\nabla g_i(x_{i,k};\zeta_{i,k})$, $\nabla^2 g_{i,k}=\nabla^2 g_i(x_{i,k};\zeta_{i,k})$.
		\item[]\textbf{End for}.
		\item[]\textbf{Output}: $\mathbf{H}_{i,k},~\mathbf{S}_{i,k}$.
	\end{itemize}
\end{algorithm}
	
Iterates $\mathbf{S}_{i,k}$ and $\mathbf{H}_{i,k}$   in (\ref{S})-(\ref{est-H}) are the estimators of $\mathbf{S}$ and $\mathbf{H}$ respectively and then $\mathbf{H}_{i,k}^{-1}\mathbf{S}_{i,k}\mathbf{H}_{i,k}^{-1}$ is the estimator of $\mathbf{H}^{-1}\mathbf{S}\mathbf{H}^{-1}$.
In  Step 1, we estimate the nonnegative left eigenvector $u/n$ of weight matrix $\mathbf{A}$ by vector $[u_{1,k}(1),u_{2,k}(2),$ $\cdots,u_{n,k}(n)]^\intercal$, which will be used to  eliminate the error induced by weight matrix $\mathbf{A}$.
In Step 2, the rescaling 
technology  \cite[Algorithm 1]{MAI201994} is used to eliminate the error, that is, dividing $u_{i,k}(i)$ in the second term of right hand side of equalities (\ref{S})-(\ref{est-H}).
Moreover,  the first terms on the right hand of (\ref{S}) and (\ref{est-H}) are the aggregate steps to collect global information, while the second terms aim at estimating the covariance matrix and Hessian matrix with local information.
\begin{thm}\label{thm:PI-rate}
	Suppose that (a) the conditions of Theorem \ref{thm:asym norm} hold, (b) graph $\mathcal{G}_A$ is strong connected, $\mathbf{A}$ is row stochastic and $a_{ii}>0$ for any $i\in\mathcal{V}$, (c) $\Cov\left(\sum_{j=1}^n\nabla g_j(x^*;\zeta_{j})\right)$ has the separable structure, that is,
	\begin{equation*}
	\Cov\left(\sum_{j=1}^n\nabla g_j(x^*;\zeta_{j})\right)=\sum_{j=1}^n\Cov\left(\nabla g_j(x^*;\zeta_{j})\right),
	\end{equation*} 
	(d) for any $i\in\mathcal{V}$, there exist positive scalars $L_1$ and $L_2$ such that
	\begin{equation*}
	\begin{aligned}
	&\mathbb{E}\left[\left\|\nabla^2 g_i(x;\zeta_{i})-\nabla^2 g_i(x^*;\zeta_{i})\right\|\right] \le L_1\|x-x^*\|,\\
	&\mathbb{E}[\left\|\nabla^2 g_i(x^*;\zeta_{i})\right\|^2]\le L_2.
	\end{aligned}
	\end{equation*}
	Then  $$\mathbf{H}_{i,k}^{-1}\mathbf{S}_{i,k}\mathbf{H}_{i,k}^{-1}\longrightarrow\mathbf{H}^{-1}\mathbf{S}\mathbf{H}^{-1},\quad \forall i\in\mathcal{V},$$
	in probability.
\end{thm}

\begin{proof}
A similar result has been presented in  \cite{zhao2021confidence},	we mimic the proofs of \cite[Theorem 3, Theorem 4]{zhao2021confidence} to show Theorem \ref{thm:PI-rate}	in three steps.
	
	\textbf{Step 1.} To facilitate analysis, we define a matrix norm first: for any $\mathbf{M}\in\mathbb{R}^{(nd)\times d}$,
	\begin{equation*}
	\|\mathbf{M}\|_{\A c}\define\sqrt{\|m_1\|_\A^2+\|m_2\|_\A^2+\cdots+\|m_d\|_\A^2}.
	\end{equation*}	
	where $m_1,m_2,\cdots,m_d$ are the columns of $\mathbf{M}$.
	It is easy to verify that, for any $\mathbf{D}\in\mathbb{R}^{(nd)\times (nd)}$,
	\begin{align}
	\label{mat-norm-0}&\|\mathbf{D}\mathbf{M}\|_{\A c}\le \|\mathbf{D}\|_\A\|\mathbf{M}\|_{\A c},
	\end{align}
	and
	\begin{equation}\label{mat-norm-2}
	\|\mathbf{M}\|_{\A c}\le \X\|\mathbf{M}\|_F,
	\end{equation}
	where $\|\cdot\|_F$ is Frobenius norm and $\X$ is defined in (\ref{norm-bound-1}).
	Define
	\begin{equation*}
	\mathbf{S}_{k}\define \left[\mathbf{S}_{1,k}^\intercal,\mathbf{S}_{2,k}^\intercal,\cdots,\mathbf{S}_{n,k}^\intercal\right]^\intercal,\quad\bar{\mathbf{S}}_{k}\define \sum_{j=1}^n\frac{u(j)}{n}\mathbf{S}_{j,k},
	\end{equation*}	
	 where $u(j)$ is the $j$-th component of vector $u$. 
	We have
	\begin{equation*}
	\mathbf{S}_k=\frac{k}{k+1}\tilde{\mathbf{A}}\mathbf{S}_{k-1}+\frac{1}{k+1}\mathbf{W}_{k},\quad \bar{\mathbf{S}}_{k}=\left(\frac{u^\intercal}{n}\otimes\mathbf{I}_d\right)\mathbf{S}_k,
	\end{equation*}
	where
	{\small\begin{align*}
	&\mathbf{W}_{k}\define\left[\mathbf{W}_{1,k}^\intercal,\mathbf{W}_{2,k}^\intercal,\cdots,\mathbf{W}_{n,k}^\intercal\right]^\intercal,\\
	&\mathbf{W}_{i,k}\define\frac{\nabla g_{i,k}\left(\nabla g_{i,k}-\nabla g_{i,k-1}\right)^\intercal+\left(\nabla g_{i,k}-\nabla g_{i,k-1}\right)\nabla g_{i,k}^\intercal}{2u_{i,k}(i)}.
	\end{align*} }
	Then
	{\small
	\begin{align}\label{S1-S2}
	&\mathbb{E}\left[\left\|\mathbf{S}_{k}-\mathbf{1}_n\otimes \bar{\mathbf{S}}_{k}\right\|_{\A c}\right]\notag\\
	&=\mathbb{E}\left[\left\|\frac{k}{k+1}\left(\tilde{\mathbf{A}}\mathbf{S}_{k-1}-\mathbf{1}_n\otimes \bar{\mathbf{S}}_{k-1}\right)\right.\right.\notag\\
	&\left.\left.\quad+\frac{1}{k+1}\left(\mathbf{I}_{dn}-\frac{\mathbf{1}_nu^\intercal}{n}\otimes\mathbf{I}_d\right)\mathbf{W}_{k}\right\|_{\A c}\right]\notag\\
	&= \mathbb{E}\left[\left\|\frac{k}{k+1}\left(\tilde{\mathbf{A}}-\frac{\mathbf{1}_nu^\intercal}{n}\otimes\mathbf{I}_d\right)\left(\mathbf{S}_{k-1}-\mathbf{1}_n\otimes \bar{\mathbf{S}}_{k-1}\right)\right.\right.\notag\\
	&\left.\left.\quad+\frac{1}{k+1}\left(\mathbf{I}_{dn}-\frac{\mathbf{1}_nu^\intercal}{n}\otimes\mathbf{I}_d\right)\mathbf{W}_{k}\right\|_{\A c}\right]\notag\\
	&\le \frac{k}{k+1}\left\|\tilde{\mathbf{A}}-\frac{\mathbf{1}_nu^\intercal}{n}\otimes\mathbf{I}_d\right\|_\A\mathbb{E}\left[\left\|\mathbf{S}_{k-1}-\mathbf{1}_n\otimes \bar{\mathbf{S}}_{k-1}\right\|_{\A c}\right]\notag\\
	&\quad+ \frac{1}{k+1}\left\|\mathbf{I}_{dn}-\frac{\mathbf{1}_nu^\intercal}{n}\otimes\mathbf{I}_d\right\|_\A\mathbb{E}\left[\left\|\mathbf{W}_{k}\right\|_{\A c}\right]\notag\\
	&\le \frac{k}{k+1}\tau_\mathbf{A}\mathbb{E}\left[\left\|\mathbf{S}_{k-1}-\mathbf{1}_n\otimes \bar{\mathbf{S}}_{k-1}\right\|_{\A c}\right]\notag\\
	&\quad+ \frac{1}{k+1}\left\|\mathbf{I}_{dn}-\frac{\mathbf{1}_nu^\intercal}{n}\otimes\mathbf{I}_d\right\|_\A\X\mathbb{E}\left[\left\|\mathbf{W}_{k}\right\|_F\right],
	\end{align}
	}
	where $\tau_\A=\left\|\tilde{\mathbf{A}}-\frac{\mathbf{1}_nu^\intercal}{n}\otimes\mathbf{I}_d\right\|_\A<1$, the first inequality and second inequality follow from (\ref{mat-norm-0}) and (\ref{mat-norm-2}) respectively. 
	By the definition of $\mathbf{W}_{k}$,
	{\footnotesize\begin{equation}\label{ie-7}
	\begin{aligned}
	&\mathbb{E}\left[\left\|\mathbf{W}_{k}\right\|_F\right]\\
	&=\mathbb{E}\left[\sqrt{\sum_{j=1}^n\left\|\frac{\nabla g_{j,k}\left(\nabla g_{j,k}-\nabla g_{j,k-1}\right)^\intercal+\left(\nabla g_{j,k}-\nabla g_{j,k-1}\right)\nabla g_{j,k}^\intercal}{2u_{j,k}(i)}\right\|_F^2}\right]\\
	&\le\mathbb{E}\left[\sqrt{\sum_{j=1}^n\left\|\frac{\nabla g_{j,k}\left(\nabla g_{j,k}-\nabla g_{j,k-1}\right)^\intercal}{u_{j,k}(i)}\right\|_F^2}\right]\\
	&\le\sum_{j=1}^n\mathbb{E}\left[\left\|\frac{\nabla g_{j,k}\left(\nabla g_{j,k}-\nabla g_{j,k-1}\right)^\intercal}{u_{j,k}(i)}\right\|_F\right]\\
	&\le \sum_{j=1}^n\frac{\mathbb{E}\left[\left\|\nabla g_{j,k}\left(\nabla g_{j,k}-\nabla g_{j,k-1}\right)^\intercal\right\|_F\right]}{\eta}\\
	&\le \sum_{j=1}^n\frac{1.5\mathbb{E}\left[\left\|\nabla g_{j,k}\right\|^2\right]+0.5\mathbb{E}\left[\left\|\nabla g_{j,k-1}\right\|^2\right]}{\eta}\\
	&\le n \sup_{j\in\mathcal{V},k\ge 0} \frac{2\mathbb{E}\left[\left\|\nabla g_{j,k}\right\|^2\right]}{\eta},
	\end{aligned}
	\end{equation}}
	where $\eta$ is some positive constant, the third inequality follows from the fact $u_{j,k}(j)\ge\eta$ \cite[Proposition 1]{MAI201994}, the fourth inequality follows from the facts $\|xy^\intercal\|_F=\|x\|\|y\|$ for any $x,y\in\mathbb{R}^d$ and $a_1a_2\le \frac{a_1^2+a_2^2}{2}$ for any $a_1,a_2\ge0$. Moreover,
	\begin{equation}\label{ie-8}
	\begin{aligned}
	&\sup_{k\ge0}\mathbb{E}\left[\left\|\nabla g_{j,k}\right\|^2\right]\\
	&\le 2\sup_{k\ge 0}\left(\mathbb{E}\left[\left\|\nabla g_j-\nabla g_{j,t}^*\right\|^2\right]+\mathbb{E}\left[\left\|\nabla g_{j,t}^*\right\|^2\right]\right)\\
	&\le 2\sup_{k\ge 0}\left(L^2\mathbb{E}\left[\left\|x_{j,k}-x^*\right\|^2\right]+\mathbb{E}\left[\left\|\nabla g_{j,t}^*\right\|^2\right]\right)\\
	&\le 2\sup_{k\ge 0}\left(2L^2\mathbb{E}\left[\left\|x_{j,k}-\bar{x}_k\right\|^2\right]+2L^2\mathbb{E}\left[\left\|\bar{x}_k-x^*\right\|^2\right]\right.\\
	&\left.\quad+\mathbb{E}\left[\left\|\nabla g_{j,t}^*\right\|^2\right]\right)\\
	&<\infty,
	\end{aligned}
	\end{equation}
	where $\nabla g_{j,t}^*\define\nabla g_j(x^*,\zeta_{j,t})$, the fourth inequality follows from the boundedness of $\mathbb{E}\left[\left\|x_{j,k}-\bar{x}_k\right\|^2\right]$, $\mathbb{E}\left[\left\|\bar{x}_k-x^*\right\|^2\right]$	in Lemma \ref{lem:rate}. Combining (\ref{ie-7}) with (\ref{ie-8}),
	we have $\mathbb{E}\left[\left\|\mathbf{W}_{k}\right\|_F\right]\le \frac{2nc_g}{\eta}$, where
	\begin{equation}\label{c_g}
	c_g\define \sup_{k\ge0}\mathbb{E}\left[\left\|\nabla g_j(x_{j,t},\zeta_{j,t})\right\|^2\right].
	\end{equation}
	 Therefore, (\ref{S1-S2}) implies following inequality recursively,
	\begin{equation*}
	\begin{aligned}
	&\mathbb{E}\left[\left\|\mathbf{S}_{k}-\mathbf{1}_n\otimes \bar{\mathbf{S}}_{k}\right\|_{\A c}\right]\\
	&\le \frac{1}{k+1}\left(\tau_\mathbf{A}^k\mathbb{E}\left[\left\|\mathbf{S}_{0}-\mathbf{1}_n\otimes \bar{\mathbf{S}}_{0}\right\|_{\A c}\right]\right.\\
	&\left.\quad+ \sum_{t=1}^k\tau_\mathbf{A}^{k-t}\left\|\mathbf{I}_{dn}-\frac{\mathbf{1}_nu^\intercal}{n}\otimes\mathbf{I}_d\right\|_\A \frac{2nc_g\X}{\eta}\right)\\
	&=\mathcal{O}\left(\frac{1}{k+1}\right).
	\end{aligned}
	\end{equation*}	
	
	\textbf{Step 2.} Define 
	\begin{align*}
	&\mathbf{\Lambda}_k\define\frac{1}{k+1}\sum_{t=0}^k\sum_{j=1}^n\left(\frac{\nabla g_{j,k}^*\left(\nabla g_{j,k}^*-\nabla g_{j,k-1}^*\right)^\intercal}{2}\right.\\
	&\left.\quad\quad+\frac{\left(\nabla g_{j,k}^*-\nabla g_{j,k-1}^*\right)\nabla g_{j,k}^{*\intercal}}{2}\right),
	\end{align*}
	where $\nabla g_{j,t}^*\define\nabla g_j(x^*,\zeta_{j,t})$.
	Note that 
	\begin{align*}
	\bar{\mathbf{S}}_{k}&=\frac{k}{k+1}\bar{\mathbf{S}}_{k-1}+\frac{1}{k+1}\left(\frac{u^\intercal}{n}\otimes\mathbf{I}_d\right)\mathbf{W}_{k}\\
	&=\frac{1}{k+1}\sum_{t=0}^k\sum_{j=1}^n\frac{u_j}{n}\left(\frac{\nabla g_{j,t}\left(\nabla g_{j,t}-\nabla g_{j,t-1}\right)^\intercal}{2u_{j,t}(i)}\right.\\
	&\left.\quad+\frac{\left(\nabla g_{j,t}-\nabla g_{j,t-1}\right)\nabla g_{j,t}^\intercal}{2u_{j,t}(i)}\right).
	\end{align*}
    We have
	{\small\begin{equation}\label{S-S bound}
	\begin{aligned}
	&\mathbb{E}\left[\left\|\bar{\mathbf{S}}_{k}-\mathbf{\Lambda}_k\right\|_F\right]\\
	&\le\mathbb{E}\left[\left\|\frac{1}{k+1}\sum_{t=0}^k\sum_{j=1}^n\left[\frac{u_j}{n}\frac{\nabla g_{j,t}\left(\nabla g_{j,t}-\nabla g_{j,t-1}\right)^\intercal}{u_{j,t}(i)}\right.\right.\right.\\
	&\left.\left.\left.\quad-\nabla g_{j,t}^*\left(\nabla g_{j,t}^*-\nabla g_{j,t-1}^*\right)^\intercal\right]\right\|_F\right]\\
	&\le\mathbb{E}\left[\left\|\frac{1}{k+1}\sum_{t=0}^k\sum_{j=1}^n\left[\frac{u_j}{nu_{j,t}(i)}-1\right]\nabla g_{j,t}\left(\nabla g_{j,t}-\nabla g_{j,t-1}\right)^\intercal\right\|_F\right]\\
	&\quad+\mathbb{E}\left[\left\|\frac{1}{k+1}\sum_{t=0}^k\sum_{j=1}^n\left[\nabla g_{j,t}\left(\nabla g_{j,t}-\nabla g_{j,t-1}\right)^\intercal\right.\right.\right.\\
	&\left.\left.\left.\quad-\nabla g_{j,t}^*\left(\nabla g_{j,t}^*-\nabla g_{j,t-1}^*\right)^\intercal\right]\right\|_F\right].
	\end{aligned}
	\end{equation}}
	For the first term on the right hand side of above inequality,
	{\small\begin{align}\label{si-ie-1}
	&\mathbb{E}\left[\left\|\frac{1}{k+1}\sum_{t=0}^k\sum_{j=1}^n\left[\frac{u_j}{nu_{j,t}(i)}-1\right]\nabla g_{j,t}\left(\nabla g_{j,t}-\nabla g_{j,t-1}\right)^\intercal\right\|_F\right]\notag\\
	&\le \frac{1}{k+1}\sum_{t=0}^k\sum_{j=1}^n\left|\frac{u_j}{nu_{j,t}(i)}-1\right|\left(1.5\mathbb{E}\left[\left\|\nabla g_{j,t}\right\|^2\right]\right.\notag\\
	&\left.\quad+0.5\mathbb{E}\left[\left\|\nabla g_{j,t-1}\right\|^2\right]\right)\notag\\
	&\le \frac{1}{k+1}\sum_{t=0}^k\sum_{j=1}^n2\left|\frac{u_j}{nu_{j,t}(i)}-1\right|c_g\notag\\
	&\le \frac{1}{k+1}\sum_{t=0}^k\frac{2nc_gc\lambda^k}{\eta}\le \frac{1}{k+1}\frac{2nc_gc}{(1-\lambda)\eta},
	\end{align}}
	where $c_g$ is defined in (\ref{c_g}), constants $c,\lambda$ and $\eta$ satisfy $c>0$, $0<\lambda<1$ and $\eta>0$ respectively, the first inequality follows from the facts $\|xy^\intercal\|_F=\|x\|\|y\|$ for any $x,y\in\mathbb{R}^d$ and $a_1a_2\le \frac{a_1^2+a_2^2}{2}$ for any $a_1,a_2\ge0$, the third inequality follows from the fact $\left|u_{j,k}(j)-\frac{u_j}{n}\right|\le c\lambda^k$ and $u_{j,k}(j)\ge\eta$ \cite[Proposition 1]{MAI201994}. For the second term on the right hand side of (\ref{S-S bound}),
	{\footnotesize\begin{align}\label{si-ie-2}
	&\mathbb{E}\left[\left\|\frac{1}{k+1}\sum_{t=0}^k\sum_{j=1}^n\left[\nabla g_{j,t}\left(\nabla g_{j,t}-\nabla g_{j,t-1}\right)^\intercal\right.\right.\right.\notag\\
	&\left.\left.\left.\quad-\nabla g_{j,t}^*\left(\nabla g_{j,t}^*-\nabla g_{j,t-1}^*\right)^\intercal\right]\right\|_F\right]\notag\\
	&\le\frac{1}{k}\sum_{t=1}^{k}\sum_{j=1}^n\mathbb{E}\left[\left\|\left(\nabla g_{j,t}-\nabla g_{j,t}^*\right)\left(\nabla g_{j,t}-\nabla g_{j,t-1}\right)^\intercal \right.\right.\notag\\
	&\left.\left.\quad-\nabla g_{j,t}^*\left(\nabla g_{j,t}^*-\nabla g_{j,t}+\nabla g_{j,t-1}-\nabla g_{j,t-1}^*\right)^\intercal\right\|_F\right]\notag\\
	&\le\frac{1}{k}\sum_{t=1}^{k}\sum_{j=1}^n\sqrt{\mathbb{E}\left[\left\|\nabla g_{j,t}-\nabla g_{j,t}^*\right\|^2\right]\mathbb{E}\left[\left\|\nabla g_{j,t}-\nabla g_{j,t-1}\right\|^2\right]}+\frac{1}{k}\sum_{t=1}^{k}\sum_{j=1}^n\notag\\
	&\quad \sqrt{2\mathbb{E}\left[\left\| \nabla g_{j,t}^*\right\|^2\right]\left(\mathbb{E}\left[\left\|\nabla g_{j,t}-\nabla g_{j,t}^*\right\|^2\right]+\mathbb{E}\left[\left\|\nabla g_{j,t-1}-\nabla g_{j,t-1}^*\right\|^2\right]\right)}\notag\\
	&\le\frac{1}{k}\sum_{t=1}^{k}\sum_{j=1}^n\sqrt{2c_gL^2\mathbb{E}\left[\left\|x_{j,t}-x^*\right\|^2\right]}+\frac{1}{k}\sum_{t=1}^{k}\sum_{j=1}^n\notag\\
	&\quad\sqrt{2\mathbb{E}\left[\left\| \nabla g_{j,t}^*\right\|^2\right]L^2\left(\mathbb{E}\left[\left\|x_{j,t}-x^*\right\|^2\right]+\mathbb{E}\left[\left\|x_{j,t-1}-x^*\right\|^2\right]\right)}\notag\\
	&=\frac{n}{k}\sum_{t=1}^{k}\mathcal{O}(\sqrt{\alpha_{t}})=\mathcal{O}\left(\frac{1}{k^{\alpha/2}}\right),
	\end{align}}
	where $c_g$ is defined in  (\ref{c_g}), the second inequality follows from H$\ddot{o}$lder inequality and the third inequality follows from the Lipschitz continuity of $\nabla g_j(\cdot;\zeta_j)$. Substituting (\ref{si-ie-1}) and (\ref{si-ie-2}) into (\ref{S-S bound}), we have
	\begin{equation*}
	\mathbb{E}\left[\left\|\bar{\mathbf{S}}_{k}-\mathbf{\Lambda}_k\right\|_F\right]= \mathcal{O}\left(\frac{1}{k^{\alpha/2}}\right).
	\end{equation*}
	By the same analysis in the step 3  of \cite[Theorm 3]{zhao2021confidence},
		$	\mathbb{E}\left[\left\|\mathbf{\Lambda}_k-\mathbf{S}\right\|\right]=\mathcal{O}\left(\frac{1}{k^{\alpha/2}}\right)$,
	which implies
	\begin{equation}\label{s-c}
	\mathbb{E}\left[\left\|\mathbf{S}_{i,k}-\mathbf{S}\right\|\right]=\mathcal{O}\left(\frac{1}{k^{\alpha/2}}\right).
	\end{equation}

	\textbf{Step 3.} Similar to the analysis of (\ref{s-c}), we may show that 
	\begin{equation}\label{h-c}
	\mathbb{E}\left[\left\|\mathbf{H}_{i,k}-\mathbf{H}\right\|\right]=\mathcal{O}\left(\frac{1}{k^{\alpha/2}}\right).
	\end{equation}
	 Combining (\ref{s-c}), (\ref{h-c}) with \cite[Corollary 4.3]{chen2020statistical}, we have that $$\mathbf{H}_{i,k}^{-1}\mathbf{S}_{i,k}\mathbf{H}_{i,k}^{-1}\longrightarrow\mathbf{H}^{-1}\mathbf{S}\mathbf{H}^{-1},\quad \forall i\in\mathcal{V},$$
	in probability.
The proof is complete.

\end{proof}

\section{Experimental Results}\label{num-exm}
In this section, we perform a simulation study to illustrate our theoretic findings about convergence rate and asymptotic normality of $\mathcal{S}$-$\mathcal{AB}$.  
Consider the ridge regression problem \cite[Part V]{pu2020distributed}:
\begin{equation}\label{sim-pro}
\min_{x\in\mathbb{R}^d} ~f(x)=\sum_{j=1}^n\left(\mathbb{E}\left[\left(w_j^\intercal x-v_j\right)^2\right]+\gamma\|x\|^2\right),\\
\end{equation}
where $f_i(x)\define\mathbb{E}\left[\left(w_i^\intercal x-v_i\right)^2\right]+\gamma\|x\|^2$ is the objective function of agent $i$.
In problem (\ref{sim-pro}),
each agent $ i\in\mathcal{V}$ has access to sample $(w_i,v_i)$ given by the following linear model
$$v_i=w_i^\intercal \tilde{x}_i+\nu_i,$$
where $ w_i $ is the regression vector,  $\nu_i $ is the observation noise and $\tilde{x}_i $ is the unknown parameter. Assume that random variables $w_i $ and $\nu_i$ are independent and then the  unique solution to problem (\ref{sim-pro}) is $$x^*=\left(\sum_{j=1}^n\mathbb{E}[w_jw_j^\intercal]+n\gamma\mathbf{I}\right)^{-1}\sum_{j=1}^n\mathbb{E}[w_jw_j^\intercal]\tilde{x}_j.$$

In this experiment, the setting of parameters follows the setting in \cite[Part V]{pu2020distributed}, the setting of network topology and weighted matrices follow from \cite[Part V]{pu2020robust}. We suppose $\gamma=1$, random variables $w_i \in [1, 2]^3$ is uniformly distributed, $\nu_i$ is drawn from the Gaussian distribution $N(0,1)$ and parameter $\tilde{x}_i$ is evenly located in $[1,~10]^3$ for $\forall i\in\mathcal{V}=\{1,2,...,20\}$.  The directed graph $\mathcal{G}$ is generated by adding random links to a ring network, where a directed link exists between any two nonadjacent nodes with a probability $p=0.3$. 
For $\forall i\in \mathcal{V}$,
$\mathcal{G}_\A=\mathcal{G}_\Z=\mathcal{G}$ and
\begin{equation*}
\begin{aligned}
&\mathbf{A}_{ij}=\left\{
\begin{aligned}
&\frac{1}{|\mathcal{N}_{\mathbf{A},i}^{\text{in}}|+1},\quad j\in \mathcal{N}_{\mathbf{A},i}^{\text{in}},\\
&1-\sum_{j\in\mathcal{N}_{\mathbf{A},i}^{\text{in}}}\mathbf{A}_{ij},\quad j=i,
\end{aligned}\right.
\\
&\mathbf{B}_{ji}=\left\{
\begin{aligned}
&\frac{1}{|\mathcal{N}_{\mathbf{B},i}^{\text{out}}|+1},\quad j\in\mathcal{N}_{\mathbf{B},i}^{\text{out}},\\
&1-\sum_{j\in\mathcal{N}_{\mathbf{B},i}^{\text{out}}}\mathbf{B}_{ji},\quad j=i,
\end{aligned}\right.
\end{aligned}
\end{equation*}
where $\mathcal{N}_{\mathbf{A},i}^{\text{in}}\define\{j|(j,i)\in\mathcal{E}_\A,j\ne i\}$ and $\mathcal{N}_{\mathbf{B},i}^{\text{out}}\define\{j|(i,j)\in\mathcal{E}_\Z,j\ne i\}$ are sets of incoming and outgoing neighbors  of agent $i$ respectively, $|\mathcal{N}_{\mathbf{A},i}^{\text{in}}|$ and $|\mathcal{N}_{\mathbf{B},i}^{\text{out}}|$ are the  cardinality of $\mathcal{N}_{\mathbf{A},i}^{\text{in}}$ and $\mathcal{N}_{\mathbf{B},i}^{\text{out}}$.

\begin{figure}[http]
	\centering
	\includegraphics[height=2.1in,width=3.2in]{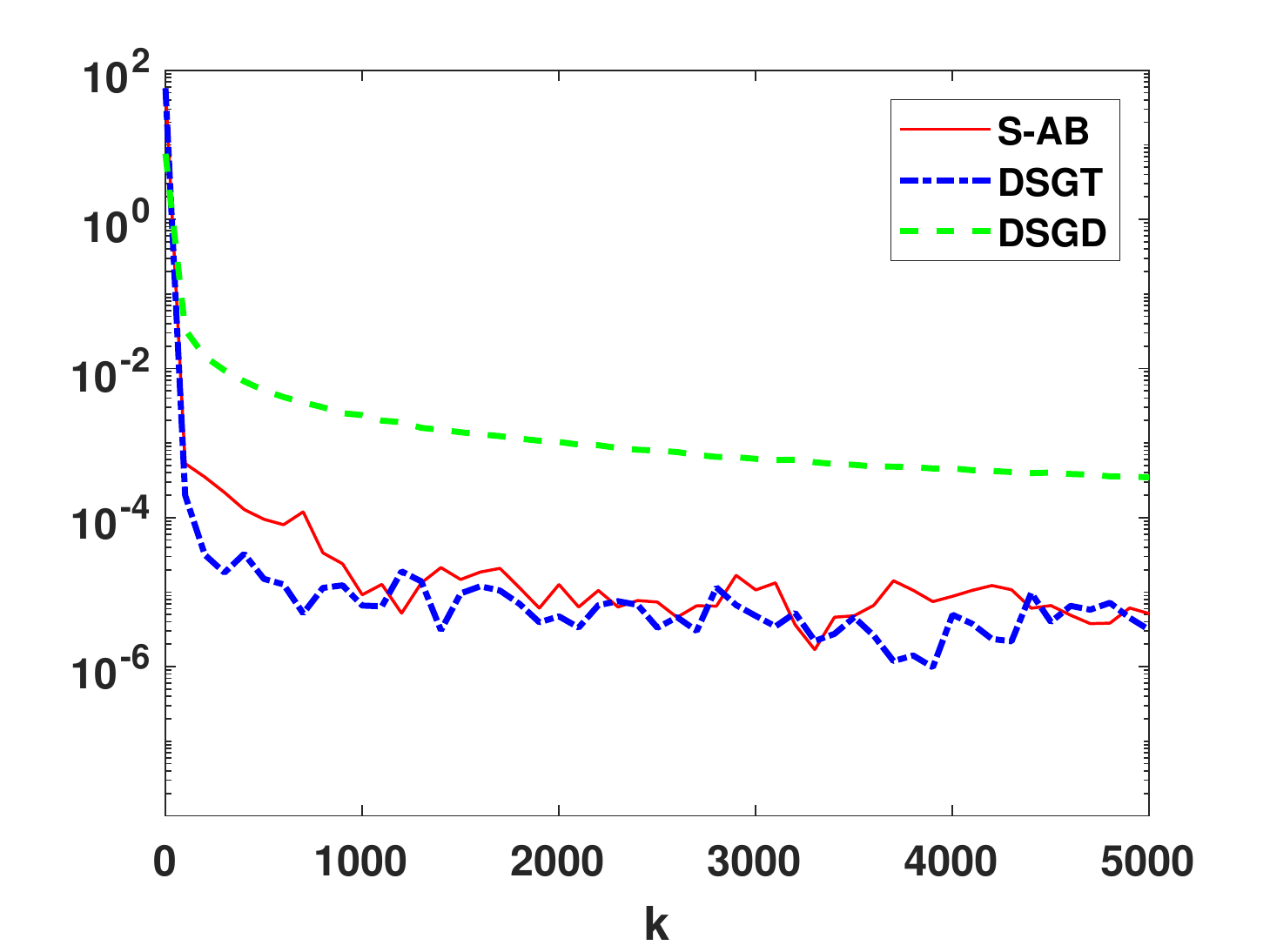}
	\caption{ Trajectories of $\frac{1}{n}\sum_{j=1}^n\|x_{j,k}-x^*\|^2$. }
	\label{fig:convergence}
\end{figure}

In Figure \ref{fig:convergence}, we report the  performances of $\mathcal{S}$-$\mathcal{AB}$, distributed stochastic gradient tracking method (DSGT) and  distributed stochastic gradient descent method (DSGD) under  the setting that step-size $\al_k=0.05/(k+1)^{0.6}$ and $x_{i,0}=\mathbf{0}$ for $\forall i\in\mathcal{V}$.  \footnote{DSGT and  DSGD employ a same doubly stochastic weight matrix, which is defined by the Metropolis rule based on the underlying graph of $\mathcal{G}$.} We run the simulations 50 times for $\mathcal{S}$-$\mathcal{AB}$, DSGT and  DSGD and average the results to approximate the expected errors after 5000 iterations,  where the solid curve, dash-dot curve and dashed curve display the average $\frac{1}{n}\sum_{i=1}^n\|x_{i,k}-x^*\|^2$ for $\mathcal{S}$-$\mathcal{AB}$, DSGT and DSGD respectively. 
From Figure \ref{fig:convergence},  we may conclude  that the   gradient tracking based method, $\mathcal{S}$-$\mathcal{AB}$ and DSGT, outperform DSGD. On the other hand,   $\mathcal{S}$-$\mathcal{AB}$ is lightly slower  than  DSGT.  However  DSGT and DSGD   require  the weight matrices are doubly stochastic and communication networks are undirected.
We also   test  $\mathcal{S}$-$\mathcal{AB}$, DSGT and  DSGD on the  hand-written classifying problem \cite[Part V]{RX2019} (MNIST dataset \cite{lecun1998mnist}) and record the  performances of the three methods in Figure \ref{fig:convergence-1}. Obviously, Figure \ref{fig:convergence-1} supports the conclusion of Figure \ref{fig:convergence}.\footnote{The step-size $\al_k=0.0001/(k+1)^{0.505}$, initial point $x_{i,0}=\mathbf{0}$ for $\forall i\in\mathcal{V}$ and     the number  of iterations $k=5000$.}

\begin{figure}[http]
	\centering
	\includegraphics[height=2.1in,width=3.2in]{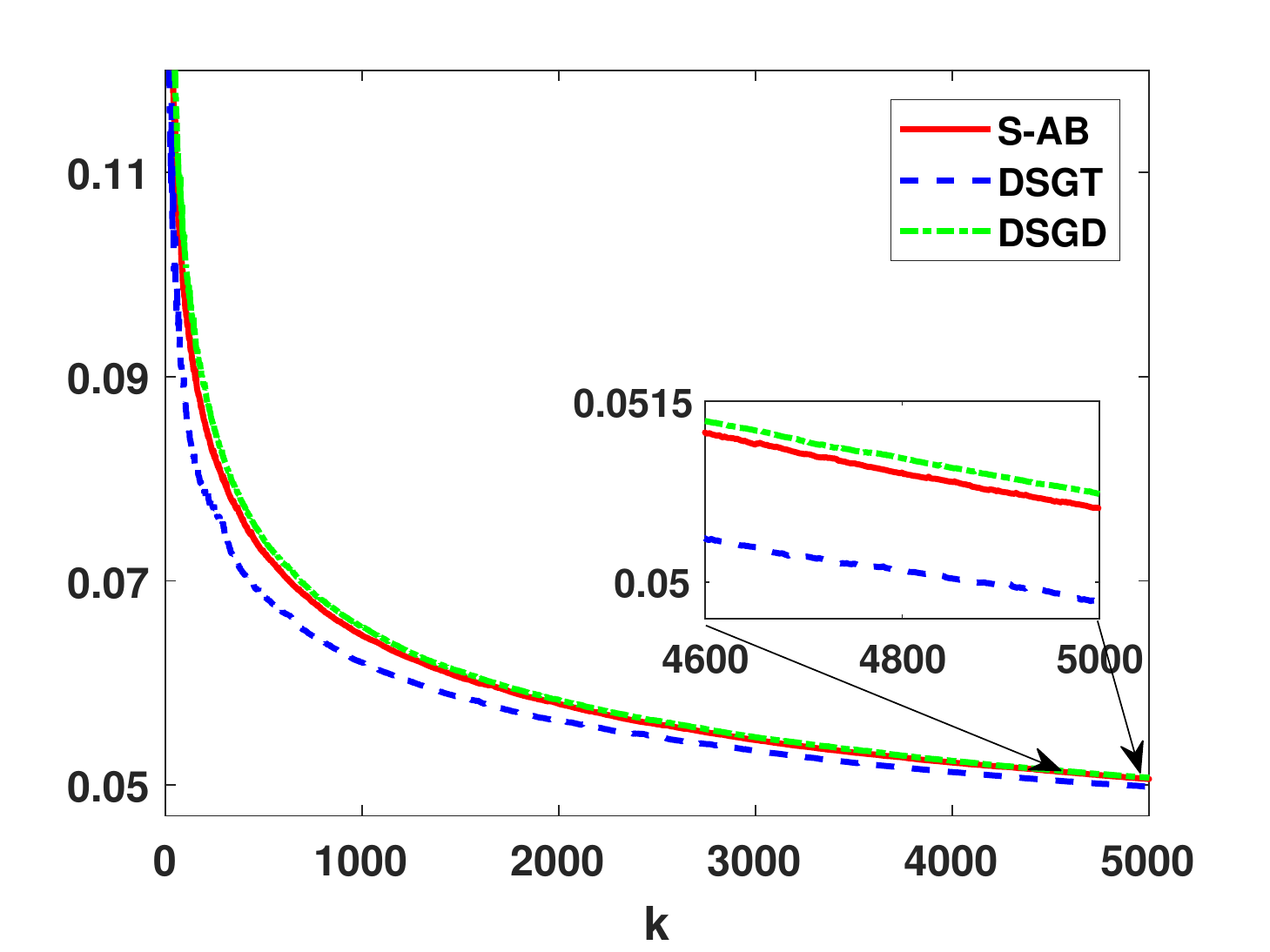}
	\caption{ Trajectories of $\frac{1}{n}\sum_{j=1}^nf(x_{j,k})$. }
	\label{fig:convergence-1}
\end{figure}

We carry out tests on the asymptotic normality of $\mathcal{S}$-$\mathcal{AB}$ algorithm for solving problem (\ref{sim-pro}).  We do 500 Monte-Carlo simulations of running $\mathcal{S}$-$\mathcal{AB}$ algorithm with 30000 iterations. In Figure \ref{fig:asym norm,effi}, the black solid curve and red circle curve denote the estimated density of agent 1 and the average of all agents respectively,  the blue  dash-dot curve denotes the true density. The simulation results shown in Figure \ref{fig:asym norm,effi} are consistent with Theorem \ref{thm:asym norm} since we can see that the estimated density of a component of normalized estimation error is close to the density of the limiting normal distribution,  which is also confirmed by a Kolmogorov–Smirnov test.

 \begin{figure}[http]
 		\begin{minipage}[t]{0.4\textwidth}
 			\centering
 			\includegraphics[height=2.1in,width=3.2in]{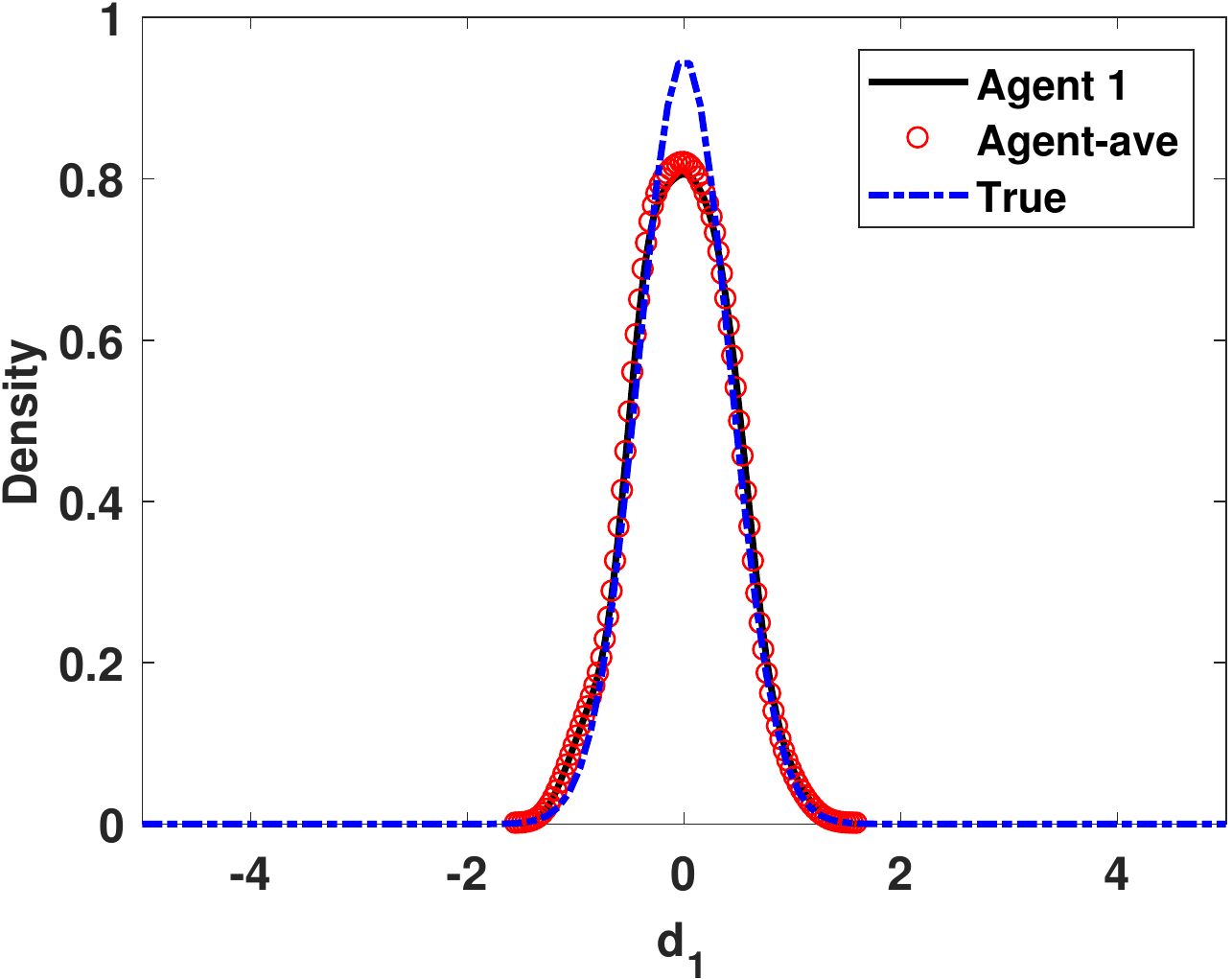}	
 			\includegraphics[height=2.1in,width=3.2in]{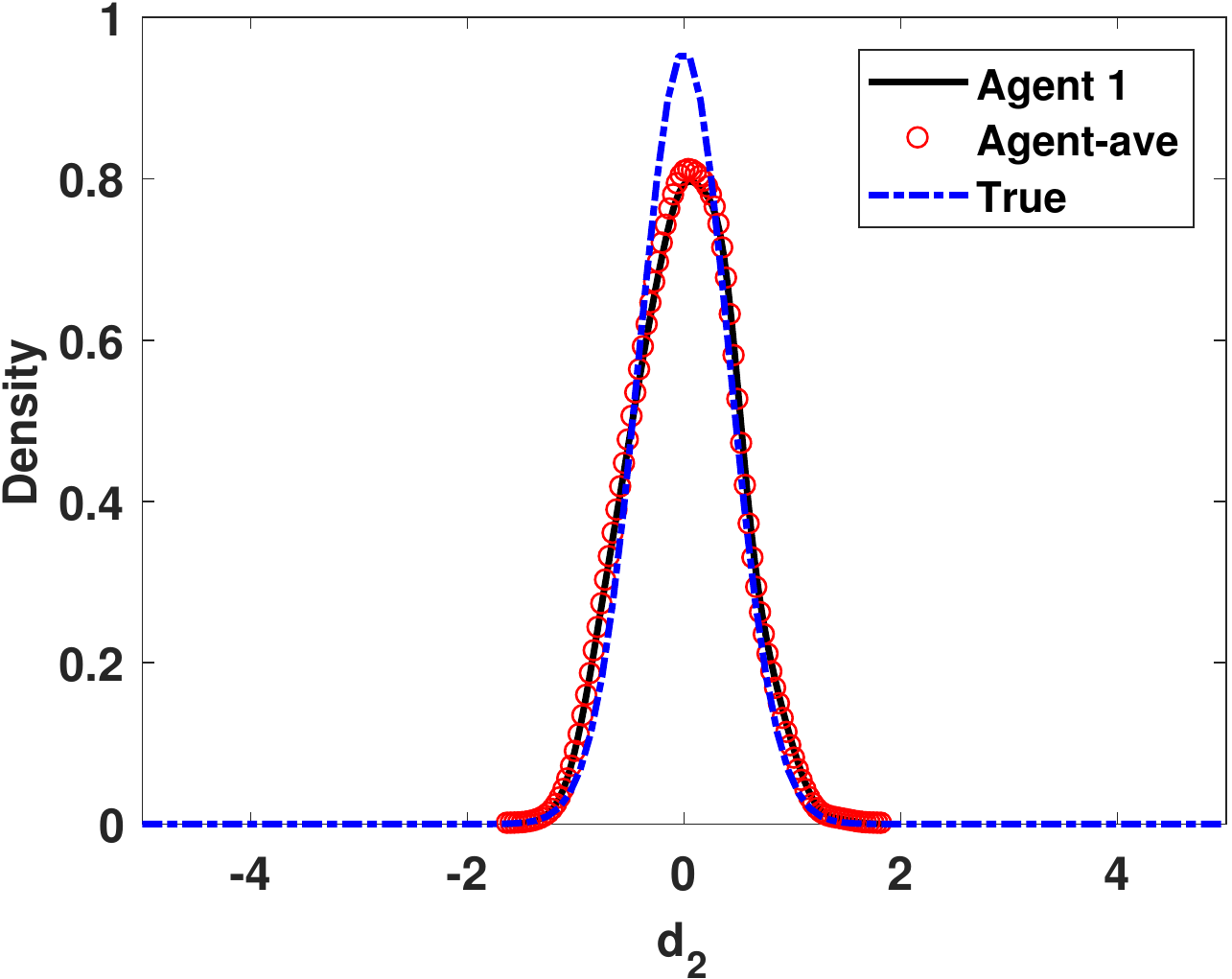}	
 			\includegraphics[height=2.1in,width=3.2in]{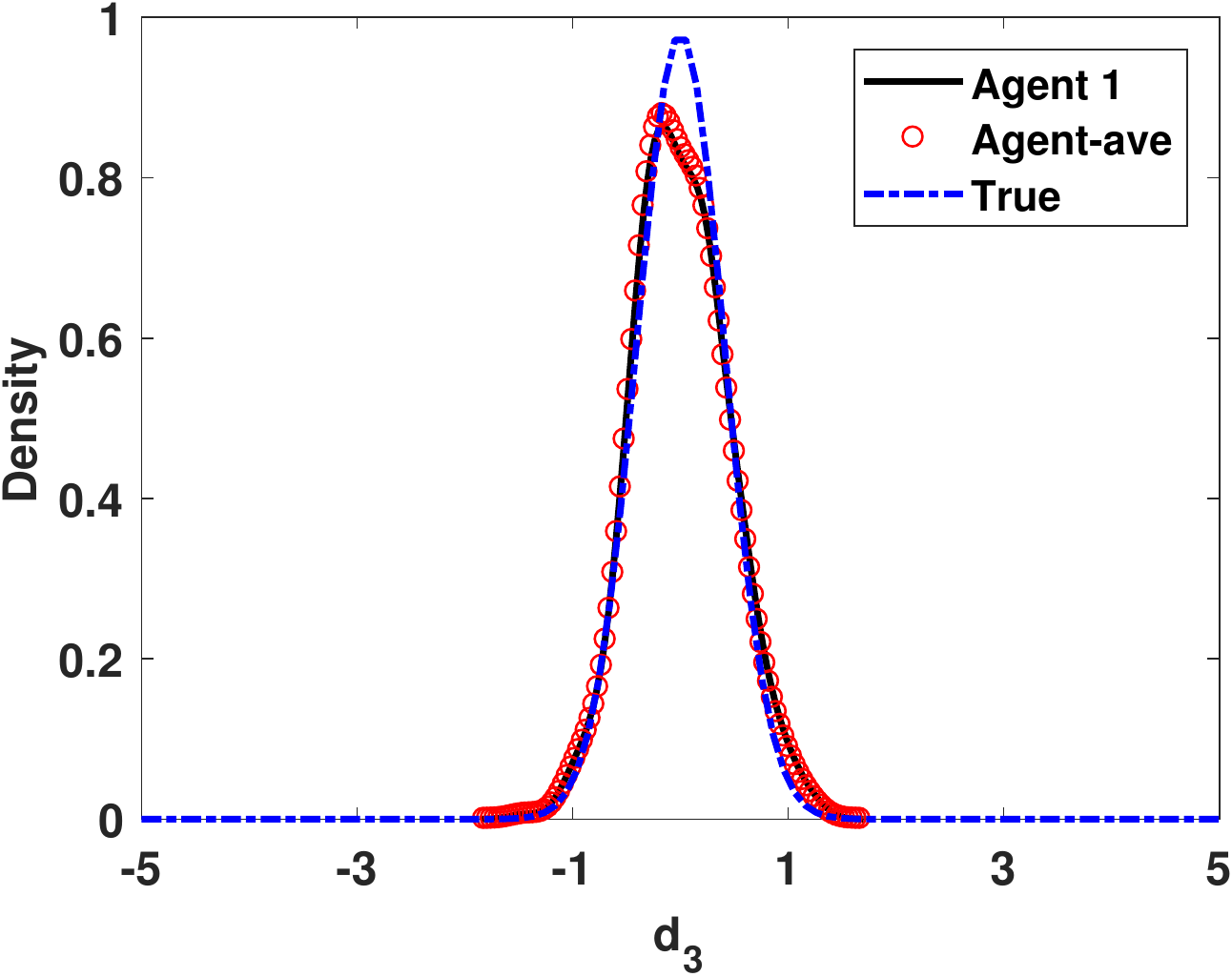}
 		\end{minipage}
 	\caption{Density of comments. }
 	\label{fig:asym norm,effi}
\end{figure}

\begin{table}[http]
	\centering
	\caption{Coverage rate($\%$).}\label{tab:cove rate}
	\begin{tabular}{l|l|l|l|l}
		\hline
		\diagbox{\centering {\small Methods}}{{\small Iterations}}& 2000& 5000& 15000&30000\\
		\hline
		PI&81&86.8&92&94.6\cr
		PIave&78.4&86.8&91.4&94.6\cr
		\hline
	\end{tabular}
\end{table}
Next, we employ the  asymptotic normality of $\mathcal{S}$-$\mathcal{AB}$ (Theorem \ref{thm:asym norm}) to construct the   confidence regions of the true solution to problem (\ref{sim-pro}). According to  Theorem \ref{thm:asym norm} and Theorem \ref{thm:PI-rate},
\begin{equation}\label{confidence region}
\left\{y:\left(y-\hat{x}\right)^\intercal\hat{\Sigma}^{-1}\left(y-\hat{x}\right)\le \frac{1}{k}\chi_\beta^2(d)\right\}
\end{equation}
defines an $1-\beta$  asymptotic confidence region for  the optimal solution to problem (\ref{sim-pro}), where  $\hat{\Sigma}$ is the covariance matric returned by distributed plug-in method, $\hat{x}$ is Polyak-Ruppert average of iterates generated by $\mathcal{S}$-$\mathcal{AB}$, $\chi_{\beta}^{2}(d)$ is defined to be the number that satisfies $P\left(U>\chi_{\beta}^{2}(d)\right)=\beta$ for a $\chi^{2}$ random variable $U$ with $d$ degrees of freedom.   We do 500 Monte-Carlo simulations and record the percentage of the $95\%$ confidence regions containing the true solution in Table \ref{tab:cove rate}, where the rows of PI, PIave record the results for agent 1 and the average of all agents respectively. From the Table \ref{tab:cove rate}, we can see that the coverage rates of PI and PIave at iterations $k=15000$ and $k=30000$ are both nearly $95\%$. Moreover, we can also observe the agreement of $\mathbf{H}_{i,k}^{-1}\mathbf{S}_{i,k}\mathbf{H}_{i,k}^{-1}$ for $i\in\mathcal{V}$ by comparing the performance of PI with PIave.

\section*{Acknowledgment}
The research is supported by the NSFC \#11971090 and  Fundamental Research Funds for the Central Universities   DUT22LAB301.

%
\bibliographystyle{IEEEtran}
\bibliography{dsi_mybib}
\section{proof of Lemma \ref{lem:ieq-1}}\label{apd:proof of Lem-ieq-1}

\begin{proof} This proof mimics the proof of \cite[Lemma 3]{bianchi2013performance}.
	Set $\tilde{\gamma}_k=(1+M)\gamma_k$. Define two sequences $\{a_{k},b_k\}_{k\ge k_0}$ such that $a_{k_0}=b_{k_0}=\max\{u_{k_0},v_{k_0}\}$ and for each $k\ge k_0+1$:
	\begin{align}
	&a_{k+1}=\rho_ka_k+\tilde{\gamma}_k\sqrt{a_k(1+a_k+b_k)}\notag\\
	\label{a}&\quad+\tilde{\gamma}_k\left(\tilde{\gamma}_k+\sum_{t=k_0}^k\tilde{\gamma}_t\rho^{k-t}a_t+\sum_{t=k_0}^k\tilde{\gamma}_t\rho^{k-t}b_t\right),\\
	&b_{k+1}=b_k+Ma_k+\tilde{\gamma}_k\sqrt{a_k(1+a_k+b_k)}\notag\\
	\label{b}&\quad+\tilde{\gamma}_k\left(\tilde{\gamma}_k+\sum_{t=k_0}^k\tilde{\gamma}_t\rho^{k-t}a_t+\sum_{t=k_0}^k\tilde{\gamma}_t\rho^{k-t}b_t\right).
	\end{align}
	It is straightforward to show by induction that $u_k\le a_k$ and $v_k\le b_k$ for any $k\ge k_0$. In addition, $b_{k+1}=b_k+a_{k+1}+(M-\rho_{k})a_{k}$ . Thus for $k\ge k_0+1$,
	\begin{align*}
	b_{k+1}&=a_{k+1}+\sum_{t=k_0}^{k}(M+1-\rho_t)a_t\\
	&=(M+1)\sum_{t=k_0}^{k+1}a_t-\sum_{t=k_0}^{k}\rho_ta_t-Ma_{k+1}.
	\end{align*}
	Define $A_k=(M+1)\sum_{t=k_0}^ka_t$. The above equality implies that $a_k\le b_k\le A_k$. As a consequence Eq (\ref{a}) implies:
	{\small\begin{equation}\label{a-1}
	a_{k+1}\le\rho_ka_k+\tilde{\gamma}_k\sqrt{a_k(1+2A_k)}+\tilde{\gamma}_k\left(\tilde{\gamma}_k+2\sum_{t=k_0}^k\tilde{\gamma}_t\rho^{k-t}A_t\right).
	\end{equation}}
	As $\{A_k\}_{k\ge k_0}$ is a positive increasing sequence, for any $k\ge k_0+1$,
	{\small\begin{equation}\label{a-2}
	\begin{aligned}
	&\frac{a_{k+1}}{A_{k+1}}\\
	&\le\rho_k\frac{a_{k}}{A_{k}}+\tilde{\gamma}_k\sqrt{\frac{a_{k}}{A_{k}}\left(\frac{1}{A_{k_0}}+2\right)}+\tilde{\gamma}_k\left(\frac{\tilde{\gamma}_k}{A_{k_0}}+2\sum_{t=k_0}^k\tilde{\gamma}_t\rho^{k-t}\right)\\
	&\le \rho_k\frac{a_{k}}{A_{k}}+\tilde{\gamma}_k\sqrt{\frac{a_{k}}{A_{k}}\left(\frac{1}{A_{k_0}}+2\right)}+\tilde{\gamma}_k^2\left(\frac{1}{A_{k_0}}+2c_\rho\right),
	\end{aligned}
	\end{equation}}
	where $c_\rho$ is a constant and the second inequality follows from the fact $\sum_{t=k_0}^k\tilde{\gamma}_t\rho^{k-t}\le c_\rho \tilde{\gamma}_k$\footnote{The fact $\sum_{t=k_0}^k\tilde{\gamma}_t\rho^{k-t}\le c_\rho \tilde{\gamma}_k$ can be obtained by the similar analysis of Lemma 3 in Appendix A.}.
	Define $B^2=\frac{1}{A_{k_0}}+2(1+c_\rho)$, and $c_k=\phi_ka_k/A_k$. By (\ref{a-2}), for any $k\ge k_0+1$,
	\begin{equation}\label{a-3}
	c_{k+1}\le\rho_k\frac{\phi_{k+1}}{\phi_{k}}c_k+B\tilde{\gamma}_k\sqrt{c_k\phi_{k+1}}\sqrt{\frac{\phi_{k+1}}{\phi_{k}}}+B^2\tilde{\gamma}_k^2\phi_{k+1},
	\end{equation}
	and under the assumption 	
	\begin{align*}
	&\limsup_{k\rightarrow\infty}\left(\gamma_{k}\sqrt{\phi_k}+\frac{\phi_{k-1}}{\phi_k}\right)<\infty,\\
	&\liminf_{k\rightarrow\infty}\left(\gamma_{k}\sqrt{\phi_k}\right)^{-1}\left(\frac{\phi_{k-1}}{\phi_k}-\rho_k\right)>0,
	\end{align*}
	there exists $k_1\ge k_0$ and a constant $C$ such that for any $k\ge k_1$,
	{\small\begin{equation}\label{condi-1}
	\sqrt{\frac{\phi_{k}}{\phi_{k+1}}}BC\left(1+BC\tilde{\gamma}_k\sqrt{\phi_{k}}\right)\le \left(\frac{\phi_{k}}{\phi_{k+1}}-\rho_k\right)\left(\gamma_{k}\sqrt{\phi_{k+1}}\right)^{-1}.
	\end{equation}}
	
	Define
	\begin{equation}\label{A}
	A:=\max\{1/C,~1/C^2, ~c_{k_1}\}.
	\end{equation}
	
	We prove by induction on $k$ that $c_k\le A$ for any $k\ge k_1$. The claim holds true for $k=k_1$ by definition of $A$. Assume that $c_{k}\le A$ for some $k\ge k_1$. Using (\ref{a-3}) and (\ref{A}), for $k+1\ge k_1$,
	\begin{equation*}
	\frac{c_{k+1}}{A}\le\rho_k\frac{\phi_{k+1}}{\phi_{k}}+\frac{B}{\sqrt{A}}\tilde{\gamma}_k\sqrt{\phi_{k+1}}\sqrt{\frac{\phi_{k+1}}{\phi_{k}}}+\frac{B^2}{A}\tilde{\gamma}_k^2\phi_{k+1},
	\end{equation*}
	By (\ref{condi-1}), the right hand side is less than one so that $c_{k+1}\le A$. This proves that $\{c_k\}_{k\ge k_0}$ is a bounded sequence.
	
	We prove that $\{A_k\}_{k\ge k_0}$ is a bounded sequence. Using the fact that $\sup_{k\ge k_0}\rho_k\le 1$, $\{A_k\}_{k\ge k_0}$ is increasing and Eq. (\ref{a-1}), it holds for $k\ge k_1+1$,
	\begin{equation*}
	\begin{aligned}
	&A_{k+1}\\
	&=A_k+a_{k+1}\\
	&\le A_k+a_k+\tilde{\gamma}_k\sqrt{a_k(1+2A_k)}+\tilde{\gamma}_k\left(\tilde{\gamma}_k+2\sum_{t=k_0}^k\tilde{\gamma}_t\rho^{k-t}A_t\right)\\
	&\le A_k+a_k+\tilde{\gamma}_k\sqrt{a_kA_k\left(\frac{1}{A_{k_0}}+2\right)}\\
	&\quad+\tilde{\gamma}_k\left(\frac{\tilde{\gamma}_k}{A_k}+2\sum_{t=k_0}^k\tilde{\gamma}_t\rho^{k-t}\frac{A_t}{A_k}\right)A_k\\
	&\le A_k+a_k+\tilde{\gamma}_k\sqrt{a_kA_k\left(\frac{1}{A_{k_0}}+2\right)}+\tilde{\gamma}_k^2\left(\frac{1}{A_{k_0}}+2c_\rho\right)A_k\\
	&\le \left(1+c_k\phi_{k}^{-1}+B\tilde{\gamma}_k\sqrt{c_k\phi_k^{-1}}+B^2\tilde{\gamma}_k^2\right)A_k.
	\end{aligned}
	\end{equation*}
	
	Finally, since  $\sup_{k\ge k_0}c_k\le A$ and $1+t^2\le\exp(t^2)$, there exists $D>0$ s.t. for any $k+1\ge k_1+1$, $A_k\le \exp(D(\phi_k^{-1}+\tilde{\gamma}_k^2))A_k$, (note that under ($\limsup_{k\rightarrow\infty}(\tilde{\gamma}_k/\sqrt{\phi_{k}})\phi_k<\infty$), ). By assumptions, $\sum_{k=0}^\infty (\phi_k^{-1}+\tilde{\gamma}_k^2)<\infty$, $\{A_k\}_{k\ge k_0}$ is therefore bounded.
	
	The proof of the lemma is concluded upon noting that $v_k\le b_k\le A_k$ and $\phi_{k} u_k\le \phi_{k}a_k\le c_kA_k$.
	
\end{proof}

\section{Some necessary technical results of proof of Lemma \ref{lem:rate}}\label{apd:proof of Lem-rate}

\begin{lem}\label{lem:weighted-seq}
	Let  $\alpha_{k}=a/(k+b)^\alpha$, where $\alpha\in (1/2,1],a>0,b>0$. Then there exists $c_\rho>0$ such that $\sum_{t=0}^{k-1}\tau_\Z^{k-t}\alpha_t\le c_\rho \alpha_{k}$, where $\tau_\Z$ is defined in (\ref{consensus-para-1}).
\end{lem}
\begin{proof}
	Let $\beta_k=\sum_{t=0}^{k-1}\tau_\Z^{k-t}\alpha_t$, then
	\begin{equation*}
	\beta_k=\tau_\mathbf{B}\sum_{t=1}^{k-2}\tau_\mathbf{B}^{k-1-t}\alpha_t+\tau_\mathbf{B}\alpha_{k-1}=\tau_\mathbf{B} \beta_{k-1}+\tau_\mathbf{B}\alpha_{k-1}.
	\end{equation*}
	Denoting $b_k=\beta_k/\alpha_{k}$, 
	\begin{equation*}
	b_k=\tau_\mathbf{B} \frac{\alpha_{k-1}}{\alpha_{k}}b_{k-1}+\tau_\mathbf{B}\frac{\alpha_{k-1}}{\alpha_{k}}.
	\end{equation*}
	Note that $\lim_{k\rightarrow\infty}\frac{\alpha_{k-1}}{\alpha_{k}}=1$. Then there exists an integer $k_0>0$ such that $\frac{\alpha_{k-1}}{\alpha_{k}}\le \frac{1+\tau_\mathbf{B}}{2\tau_\mathbf{B}}$ for $k>k_0$. Taking
	\begin{equation*}
	c_\rho=\max\left\{\sup_{1\le k\le k_0}b_k,~\frac{1+\tau_\mathbf{B}}{1-\tau_\mathbf{B}}\right\},
	\end{equation*}
	we have  $b_k\le c_\rho$ for $k\le k_0$. Suppose
	that the claim holds for $k-1$, that is $b_{k-1}\le c_\rho$, then
	\begin{equation*}
	\begin{aligned}
	b_k=\tau_\mathbf{B} \frac{\alpha_{k-1}}{\alpha_{k}}b_{k-1}+\tau_\mathbf{B}\frac{\alpha_{k-1}}{\alpha_{k}}&\le \frac{\tau_\mathbf{B}+1}{2} (c_\rho+1)\\
	&\le \frac{c_\rho}{c_\rho+1} (c_\rho+1)=c_\rho.
	\end{aligned}
	\end{equation*}
	The proof is complete.	
\end{proof}

\begin{lem}\label{lem:noi-bound}
	Suppose that Assumptions \ref{ass:objective}-\ref{ass:matrix} hold.  Then
	\begin{itemize}
		\item [(i)] \begin{equation}\label{bound-noi}
		      \sup_{k}\mathbb{E}[\|\xi_k\|^2]\le\frac{c_b^2}{(1-\tau_\Z)^2}nU_1,
		    \end{equation}
		\item [(ii)]{\small\begin{equation}\label{cross term}
			\begin{aligned}
			&\mathbb{E}\left[\left\langle \bar{x}_{k}-x^*, \alpha_{k}\left(\frac{u^\intercal}{n}\otimes\mathbf{I}_{d}\right)\left(y_k^{'}-\frac{v}{n}\otimes \nabla f(\bar{x}_k)+\xi_k\right)\right\rangle\right]\\
			&\le \alpha_{k}\left(\mathbb{E}\left[\left\|\bar{x}_{k}-x^*\right\|^2\right]\left(2\frac{(u^\intercal vL\X)^2}{n^{3}}\mathbb{E}\left[\left\|x_k-\mathbf{1}\otimes\bar{x}_k\right\|_\A^2\right]\right.\right.\\
			&\quad\left.\left.+2\frac{\|u\|^2\X^2}{n^2}\mathbb{E}\left[\left\|y_k^{'}-v\otimes \bar{y}^{'}_k\right\|_\Z^2\right]\right)\right)^{1/2}+\alpha_{k}^2\frac{2.5\|u\|^2\X c_b^2U_1 c_\rho}{n(1-\tau_\Z)^3}\\
			&\quad+ \frac{2\|u\|^2\X\alpha_k}{n^2(1-\tau_\Z)}\sum_{t=0}^{k-1}\alpha_{t}\tau_\Z^{k-t}\bigg(\mathbb{E}\left[\X^2\|y_t^{'}-v\otimes \bar{y}^{'}_t\|_\mathbf{B}^2\right]\\
			&\quad+\frac{\|v\|^2L^2\X^2}{n}\mathbb{E}\left[\|x_t-\mathbf{1}\otimes\bar{x}_t\|_\mathbf{A}^2\right]+\frac{\|v\|^2L^2}{n^2}\mathbb{E}\left[\|\bar{x}_t-x^*\|^2\right]\bigg),
			\end{aligned}
			\end{equation}}
	\end{itemize}
	where $c_b=\max\left\{\X^2,\frac{\left\|\mathbf{B}-\mathbf{I}_{n}\right\|}{\tau_\Z}\X^2\right\}$, $\bar{c}$ is defined in (\ref{norm-bound-1}), $U_1$ is defined in Assumption \ref{ass:objective}, $c_\rho$ is defined in Lemma \ref{lem:weighted-seq}.
\end{lem}
\begin{proof}
	By the definition of $\xi_k$, we have 
	\begin{equation*}
	\xi_k=\sum_{t=0}^{k-1}\tilde{\mathbf{B}}^{k-1-t}(\tilde{\mathbf{B}}-\mathbf{I}_{nd})\epsilon_t+\epsilon_{k}=\sum_{t=0}^{k}\tilde{\mathbf{B}}(k,t)\epsilon_t,
	\end{equation*}
	where
	\begin{equation}\label{mat-B}
	\begin{aligned}
	&\tilde{\mathbf{B}}(k,t)\define \tilde{\mathbf{B}}^{k-1-t}(\tilde{\mathbf{B}}-\mathbf{I}_{nd})~(t<k),\\
	&\tilde{\mathbf{B}}(k,k-1)\define\tilde{\mathbf{B}}-\mathbf{I}_{nd},\quad\tilde{\mathbf{B}}(k,k)\define\mathbf{I}_{nd}.
	\end{aligned}
	\end{equation}
	Then we have
	\begin{equation*}
	\begin{aligned}
	&\mathbb{E}[\|\xi_k\|^2]\\
	&\le \X^2\mathbb{E}[\|\xi_k\|_\Z^2]\\
	&\le\X^2\sum_{t_1=0}^{k}\sum_{t_2=0}^{k}\left\|	\tilde{\mathbf{B}}(k,t_1)\right\|_\Z\left\|	\tilde{\mathbf{B}}(k,t_2)\right\|_\Z\mathbb{E}\left[\|\epsilon_{t_1}\|_\Z\|\epsilon_{t_2}\|_\Z\right]\\
	&\le\X^4\sum_{t_1=0}^{k}\sum_{t_2=0}^{k}\left\|	\tilde{\mathbf{B}}(k,t_1)\right\|_\Z\left\|	\tilde{\mathbf{B}}(k,t_2)\right\|_\Z\mathbb{E}\left[\|\epsilon_{t_1}\|\|\epsilon_{t_2}\|\right].
	\end{aligned}
	\end{equation*}
	By Assumption \ref{ass:matrix},
	\begin{equation}\label{mat-dimish}
	\begin{aligned}
	\left\|\tilde{\mathbf{B}}^{k}\left(\tilde{\mathbf{B}}-\mathbf{I}_{nd}\right)\right\|_\Z&=\left\|\left(\tilde{\mathbf{B}}-\frac{v\mathbf{1}^\intercal}{n}\otimes \mathbf{I}_d\right)\tilde{\mathbf{B}}^{k-1}\left(\tilde{\mathbf{B}}-\mathbf{I}_{nd}\right)\right\|_\Z\\
	&\le\tau_\Z \left\|\tilde{\mathbf{B}}^{k-1}\left(\tilde{\mathbf{B}}-\mathbf{I}_{nd}\right)\right\|_\Z\\
	&\le\cdots\le \tau_\Z^k \left\|\tilde{\mathbf{B}}-\mathbf{I}_{nd}\right\|_\Z,
	\end{aligned}
	\end{equation}
	where $\tau_\mathbf{B}\define\left\|\tilde{\mathbf{B}}-\frac{v\mathbf{1}^\intercal}{n}\otimes \mathbf{I}_d\right\|_\mathbf{B}$.
	Denoting $c_b=\max\left\{\X^2,\frac{\left\|\mathbf{B}-\mathbf{I}_{n}\right\|}{\tau_\Z}\X^2\right\}$, we have
	\begin{equation*}
	\|\tilde{\mathbf{B}}(k,t)\|_\Z\le
	\tau_\Z^{k-t}\max\left\{1,  \frac{\left\|\tilde{\mathbf{B}}-\mathbf{I}_{nd}\right\|_\Z}{\tau_\Z}\right\}\le \tau_\Z^{k-t} \frac{c_b}{\X^2},
	\end{equation*}
	and then
	\begin{equation}\label{ieq-5}
	\begin{aligned}
	\mathbb{E}[\|\xi_k\|^2]
	&\le c_b^2\sum_{t_1=0}^{k}\sum_{t_2=0}^{k}\tau_\Z^{2k-t_1-t_2}\mathbb{E}\left[\|\epsilon_{t_1}\|\|\epsilon_{t_2}\|\right]\\
	&\le c_b^2\sum_{t_1=0}^{k}\sum_{t_2=0}^{k}\tau_\Z^{2k-t_1-t_2}\frac{\mathbb{E}\left[\|\epsilon_{t_1}\|^2+\|\epsilon_{t_2}\|^2\right]}{2}\\
	&\le \frac{c_b^2}{(1-\tau_\Z)^2}nU_1,
	\end{aligned}
	\end{equation}
	where  the last inequality follows from Assumption \ref{ass:objective} (iii) and the fact $\tau_\Z<1$. Then (\ref{bound-noi})	holds.
	
	Next, we show (\ref{cross term}). Obviously,
	\begin{equation}\label{2-momon-0}
	\begin{aligned}
	&\mathbb{E}\left[\left\langle \bar{x}_{k}-x^*, \alpha_{k}\left(\frac{u^\intercal}{n}\otimes\mathbf{I}_{d}\right)\left(y_k^{'}-\frac{v}{n}\otimes \nabla f(\bar{x}_k)+\xi_k\right)\right\rangle\right]\\
	&= \mathbb{E}\left[\left\langle \bar{x}_{k}-x^*, \alpha_{k}\left(\frac{u^\intercal}{n}\otimes\mathbf{I}_{d}\right)\left(y_k^{'}-\frac{v}{n}\otimes \nabla f(\bar{x}_k)\right)\right\rangle\right]\\
	&\quad+\mathbb{E}\left[\left\langle \bar{x}_{k}-x^*, \alpha_{k}\left(\frac{u^\intercal}{n}\otimes\mathbf{I}_{d}\right)\xi_k\right\rangle\right].
	\end{aligned}
	\end{equation}	
	For the first term on the right hand side of (\ref{2-momon-0}),
	{\small\begin{equation}\label{2-momon-1}
		\begin{aligned}
		&\mathbb{E}\left[\left\langle \bar{x}_{k}-x^*, \alpha_{k}\left(\frac{u^\intercal}{n}\otimes\mathbf{I}_{d}\right)\left(y_k^{'}-\frac{v}{n}\otimes \nabla f(\bar{x}_k)\right)\right\rangle\right]\\
		&\le
		\mathbb{E}\left[\left\|\bar{x}_{k}-x^*\right\|\left(\left\|\frac{\alpha_{k}u^\intercal v}{n}\left(\bar{y}^{'}_k-\frac{1}{n}\nabla f(\bar{x}_k)\right)\right\|\right.\right.\\
		&\left.\left.+\left\|\alpha_{k}\left(\frac{u^\intercal}{n}\otimes\mathbf{I}_{d}\right)\left(y_k^{'}-v\otimes \bar{y}^{'}_k\right)\right\|\right)\right]\\
		&\le \mathbb{E}\left[\left\|\bar{x}_{k}-x^*\right\|\left(\frac{\alpha_{k}u^\intercal vL}{n^{1.5}}\left\|x_k-\mathbf{1}\otimes\bar{x}_k\right\|\right.\right.\\
		&\quad \left.\left.+\frac{\alpha_{k}\|u\|}{n}\left\|y_k^{'}-v\otimes \bar{y}^{'}_k\right\|\right)\right]\\
		&\le  \alpha_{k}\left(\mathbb{E}\left[\left\|\bar{x}_{k}-x^*\right\|^2\right]\left(2\frac{(u^\intercal vL\X)^2}{n^{3}}\mathbb{E}\left[\left\|x_k-\mathbf{1}\otimes\bar{x}_k\right\|_\A^2\right]\right.\right.\\
		&\quad\left.\left.+2\frac{\|u\|^2\X^2}{n^2}\mathbb{E}\left[\left\|y_k^{'}-v\otimes \bar{y}^{'}_k\right\|_\Z^2\right]\right)\right)^{1/2},
		\end{aligned}
		\end{equation}}
	where the first inequality follows from the facts $\bar{y}^{'}_k=\frac{1}{n}\sum_{j=1}^n\nabla f_j(x_{i,k})$ and
	\begin{equation*}
	\left\|\frac{1}{n}\sum_{j=1}^n\nabla f_j(x_i)-\frac{1}{n}\nabla f(y)\right\|^2\le \frac{L^2}{n}\sum_{j=1}^n\|x_i-y\|^2.
	\end{equation*}	
  It is left to estimate the upper bound of the second term on the right hand side of (\ref{2-momon-0}). By the formula (\ref{consensus-new-form}),
	\begin{equation}
	\begin{aligned}
	\bar{x}_{k}-x^*
	&=\bar{x}_{k-1}-x^*-\alpha_{k-1}\left(\frac{u^\intercal}{n}\otimes\mathbf{I}_{d}\right)y_{k-1}\\
	&=\bar{x}_{0}-x^*-\sum_{t=0}^{k-1}\alpha_t\left(\frac{u^\intercal}{n}\otimes\mathbf{I}_{d}\right)y_t.
	\end{aligned}
	\end{equation}
	Then for the second term on the right hand side of (\ref{2-momon-0}), 
	{\small\begin{equation*}
		\begin{aligned}
		&\mathbb{E}\left[\left\langle \bar{x}_{k}-x^*, \alpha_{k}\left(\frac{u^\intercal}{n}\otimes\mathbf{I}_{d}\right)\xi_k\right\rangle\right]\\
		&=\alpha_{k}\mathbb{E}\left[\left\langle -\sum_{t=0}^{k-1}\alpha_t\left(\frac{u^\intercal}{n}\otimes\mathbf{I}_{d}\right)y_t, \left(\frac{u^\intercal}{n}\otimes\mathbf{I}_{d}\right)\sum_{t=0}^{k}\tilde{\mathbf{B}}(k,t)\epsilon_t\right\rangle\right]\\
		&=\alpha_{k}\sum_{t_1=0}^{k-1}\sum_{t_2=0}^{t_1}\mathbb{E}\left[\left\langle-\alpha_{t_1}\left(\frac{u^\intercal}{n}\otimes\mathbf{I}_{d}\right)y_{t_1}, \left(\frac{u^\intercal}{n}\otimes\mathbf{I}_{d}\right)\tilde{\mathbf{B}}(k,t_2)\epsilon_{t_2}\right\rangle\right]\\
		&\le \alpha_{k}\sum_{t_1=0}^{k-1}\sum_{t_2=0}^{t_1}\alpha_{t_1}\frac{\|u\|^2\X}{n^2}\left\|\tilde{\mathbf{B}}(k,t_2)\right\|_\Z\left(\mathbb{E}\left[\|y_{t_1}\|\| \epsilon_{t_2}\|\right]\right)\\
		&\le \alpha_{k}\sum_{t_1=0}^{k-1}\sum_{t_2=0}^{t_1}\alpha_{t_1}\frac{\|u\|^2\X}{2n^2}\tau_\Z^{k-t_2}\left(\mathbb{E}\left[\|y_{t_1}\|^2\right]+nU_1\right),
		\end{aligned}
		\end{equation*}}
	where $\bar{c}$ is defined in (\ref{norm-bound-1}), the first and second equalities follow from the facts 
	\begin{align*}
		&\mathbb{E}\left[\left\langle\bar{x}_{0}-x^*, \left(\frac{u^\intercal}{n}\otimes\mathbf{I}_{d}\right)\tilde{\mathbf{B}}(k,t)\epsilon_t\right\rangle\bigg|\mathcal{F}_t\right]\\
		&=\left\langle\bar{x}_{0}-x^*,\mathbb{E}\left[ \left(\frac{u^\intercal}{n}\otimes\mathbf{I}_{d}\right)\tilde{\mathbf{B}}(k,t)\epsilon_t\bigg|\mathcal{F}_t\right]\right\rangle\\
		&=\left\langle\bar{x}_{0}-x^*,\mathbf{0}\right\rangle=0
		\end{align*}
		and
		\begin{equation*}
		\mathbb{E}\left[\left\langle y_{t_1}, \epsilon_{t_2}\right\rangle\big|\mathcal{F}_{t_2}\right]=\left\langle y_{t_1},\mathbb{E}\left[ \epsilon_{t_2}\big|\mathcal{F}_{t_2}\right]\right\rangle=\left\langle y_{t_1},\mathbf{0}\right\rangle=0\quad (t_1<t_2)
		\end{equation*}
	respectively. Therefore, 
	\begin{equation}\label{2-momon-2}
	\begin{aligned}
	&\mathbb{E}\left[\left\langle \bar{x}_{k}-x^*, \alpha_{k}\left(\frac{u^\intercal}{n}\otimes\mathbf{I}_{d}\right)\xi_k\right\rangle\right]\\
	&\le \alpha_{k}\frac{\|u\|^2\X}{2n^2(1-\tau_\Z)}\sum_{t=0}^{k-1}\alpha_{t}\tau_\Z^{k-t}\mathbb{E}\left[\|y_{t}\|^2\right]+\alpha_{k}^2\frac{\|u\|^2\X U_1 c_\rho}{2n(1-\tau_\Z)}\\
	&\le  \alpha_{k}\frac{2\|u\|^2\X}{n^2(1-\tau_\Z)}\sum_{t=0}^{k-1}\alpha_{t}\tau_\Z^{k-t}\left(\mathbb{E}\left[\|y_t^{'}-v\otimes \bar{y}^{'}_t\|^2\right]\right.\\
	&\quad+\mathbb{E}\left[\left\|v\otimes \left(\bar{y}^{'}_t-\frac{1}{n}\nabla f(\bar{x}_t)\right)\right\|^2\right]+\|v\|^2\mathbb{E}\left[\left\|\frac{1}{n}\nabla f(\bar{x}_t)\right\|^2\right]\\
	&\quad\left.+\mathbb{E}\left[\|\xi_t\|^2\right]\right)+\alpha_{k}^2\frac{ \|u\|^2\X U_1 c_\rho}{2n(1-\tau_\Z)}\\
	&\le  \alpha_{k}\frac{2\|u\|^2\X}{n^2(1-\tau_\Z)}\sum_{t=0}^{k-1}\alpha_{t}\tau_\Z^{k-t}\left(\mathbb{E}\left[\X^2\|y_t^{'}-v\otimes \bar{y}^{'}_t\|_\mathbf{B}^2\right]\right.\\
	&\quad\left.+\frac{\|v\|^2L^2}{n}\X^2\mathbb{E}\left[\|x_t-\mathbf{1}\otimes\bar{x}_t\|_\mathbf{A}^2\right]+\|v\|^2L^2\mathbb{E}\left[\|\bar{x}_t-x^*\|^2\right]\right)\\
	&\quad+\alpha_{k}^2\frac{\|u\|^2\X U_1 c_\rho}{n(1-\tau_\Z)}\left(\frac{1}{2}+\frac{2c_b^2}{(1-\tau_\Z)^2}\right)\\
	&\le  \alpha_{k}\frac{2\|u\|^2\X}{n^2(1-\tau_\Z)}\sum_{t=0}^{k-1}\alpha_{t}\tau_\Z^{k-t}\left(\mathbb{E}\left[\X^2\|y_t^{'}-v\otimes \bar{y}^{'}_t\|_\mathbf{B}^2\right]\right.\\
	&\quad\left.+\frac{\|v\|^2L^2}{n}\X^2\mathbb{E}\left[\|x_t-\mathbf{1}\otimes\bar{x}_t\|_\mathbf{A}^2\right]+\|v\|^2L^2\mathbb{E}\left[\|\bar{x}_t-x^*\|^2\right]\right)\\
	&\quad+\alpha_{k}^2\frac{2.5\|u\|^2\X c_b^2U_1 c_\rho}{n(1-\tau_\Z)^3},
	\end{aligned}
	\end{equation}
	where the first inequality follows from Lemma \ref{lem:weighted-seq}, the third inequality follows from 
	\begin{equation*}
	\mathbb{E}\left[\left\|\bar{y}^{'}_k-\frac{1}{n}\nabla f(\bar{x}_k)\right\|^2\right]\le \frac{L^2\X^2}{n}\mathbb{E}\left[\left\|x_k-\mathbf{1}\otimes\bar{x}_k\right\|_\mathbf{A}^2\right]
	\end{equation*}
	and
	\begin{equation*}
	\left\|\nabla f(\bar{x}_k)\right\|=\left\|\nabla f(\bar{x}_k)-\nabla f(x^*)\right\|\le nL \left\|\bar{x}_k-x^*\right\|,
	\end{equation*}
	the fourth inequality follows from the fact $1<\frac{c_b^2}{(1-\tau_\Z)^2}$. Summarizing (\ref{2-momon-0}), (\ref{2-momon-1}) and (\ref{2-momon-2}), we have (\ref{cross term}). The proof is complete.
\end{proof}

\section{Some necessary technical results of proof of Theorem \ref{thm:asym norm}}\label{apd:thm-asym-norm}
\begin{lem}[{\cite[Theorem 3.4.2]{chen2006stochastic}}]\label{lem:asym-norm}
	We introduce asymptotic properties of the sequence  generated by the following recursion:
	\begin{equation}\label{norm form}
	\Delta_{k+1}=\Delta_{k}+\alpha_kh(\Delta_{k})+\alpha_k \left(\mu_k+\eta_k\right).
	\end{equation}
	We need the following conditions.
	\begin{itemize}
		\item [(C0)] $\alpha_{k}>0$, $\alpha_{k}$ nonincreasingly converges to zero, $\sum_{k=0}^{\infty}\alpha_{k}<\infty$,
		$$\alpha_{k}k\longrightarrow \infty,\quad \frac{1}{\alpha_{k+1}}-\frac{1}{\alpha_{k}}\longrightarrow 0,$$
		and for some $\delta\in (0,1)$
		\begin{equation}\label{para}
		\sum_{k=0}^{\infty}\frac{\alpha_{k}^{(1+\delta)/2}}{\sqrt{k+1}}.
		\end{equation}
		\item [(C1)]There exists a continuously differentiable function $V(\cdot)$ such that
		\begin{equation*}
		h(x)^\intercal \nabla V(x)< 0, \quad \forall x\neq 0.
		\end{equation*}
		\item [(C2)]$h(\cdot)$ is measurable and locally bounded, and there exists a stable matrix $\mathbf{G}$ and a constant $c_1>0$ such that
		\begin{equation*}
		\left\|h(x)-\mathbf{G}x\right\|\le c_1\|x\|^2.
		\end{equation*}
		\item [(C3)]The noise sequences $\{\mu_k\}$ and $\{\eta_k\}$ satisfy
		\begin{equation}\label{apd:cond-1}
		\alpha_{k-1}\sum_{t=0}^{k-1}\mu_{t}\rightarrow 0,\quad \frac{1}{\sqrt{k}}\sum_{t=0}^{k-1}\mu_{t}\stackrel{d}{\rightarrow} N(\mathbf{0},\Sigma),
		\end{equation}
		\begin{equation}\label{apd:cond-2}
		\mathbb{E}\left[\|\mu_k\|^2\right]<\infty,\quad \sum_{s\ge -t}^{\infty}\|\mathbb{E}\left[\mu_t\mu_{t+s}^\intercal\right]\|\le c
		\end{equation}
		with $c$ being a constant independent of $t$, and
		\begin{equation}\label{apd:cond-3}
		\mathbb{E}\left[\|\eta_k\|^2\right]=\mathcal{O}(\alpha_{k}^{1+\delta}),
		\end{equation}
		where $\delta$ is specified in (\ref{para}).
		
	\end{itemize}
	Then $\tilde{\Delta}_{k}=\sum_{t=0}^{k-1}\Delta_{t}$ is asymptotic efficient:
	\begin{equation*}
	\sqrt{k}\tilde{\Delta}_{k}\stackrel{d}{\rightarrow} N(\mathbf{0},S),
	\end{equation*}
	where $S=\mathbf{G}^{-1}\Sigma \mathbf{G}^{-1}$.
\end{lem}

\begin{lem}\label{lem:noi-asym-norm}
	Suppose that Assumptions \ref{ass:objective} and \ref{ass:matrix} hold. Let step-size $\alpha_{k}=a/(k+b)^\alpha$, where $\alpha\in (1/2,1)$, positive scalars $a,b$ satisfy  $\frac{a}{b^\alpha}\le \min\{1,\frac{2n}{u^\intercal vL}\}$, $L\define\max_{1\leq i \leq n}\sqrt{\mathbb{E}[L_{i}^2(\zeta_i)]}$. Then
	\begin{equation}\label{aux-asym-norm}
	\frac{1}{\sqrt{k}}\sum_{t=0}^{k-1}\mu_t\stackrel{d}{\rightarrow} N\left(\mathbf{0},\frac{1}{n^2}\mathbf{S}\right),
	\end{equation}
	where
	$\mathbf{S}=\Cov\left(\sum_{j=1}^n\nabla f_j(x^*;\zeta_j)\right).$
	Moreover, for any $t\ge 0$	there exists a constant $c$ such that
	\begin{equation}\label{ieq-nois-bound-1}
	\sum_{s\ge -t}^{\infty}\left\|\mathbb{E}\left[ \mu_t\mu^\intercal_{s+t}\right]\right\|\le c.
	\end{equation}	
\end{lem}

\begin{proof}
	In what follows, we show (\ref{aux-asym-norm}) by \cite[Lemma 3.3.1]{chen2006stochastic}. Recall that
	{\small\begin{equation*}
		\frac{1}{\sqrt{k}}\sum_{t=0}^{k-1}\mu_{t}=-\frac{1}{\sqrt{k}}\sum_{t=0}^{k-1}\frac{n}{u^\intercal v}\left(\frac{u^\intercal}{n}\otimes\mathbf{I}_{d}\right)\xi_t=-\frac{1}{\sqrt{k}}\sum_{t=0}^{k-1}\mathbf{D}_{k-1,t}\epsilon_t,
		\end{equation*}}
	where $\mathbf{D}_{k-1,t}\define\frac{n}{u^\intercal v}\left(\frac{u^\intercal}{n}\otimes\mathbf{I}_{d}\right)\tilde{\mathbf{B}}^{k-1-t}.$ 
	Define
	{\small\begin{align*}
		\epsilon_t^*&=\left[\left(\nabla g_1(x^*;\zeta_{1,t})-\nabla f_1(x^*)\right)^\intercal,\left(\nabla g_2(x^*;\zeta_{2,t})-\nabla f_2(x^*)\right)^\intercal,\right.\\
		&\quad\quad\left.\cdots,\left(\nabla g_n(x^*;\zeta_{n,t})-\nabla f_n(x^*)\right)^\intercal\right]^\intercal.
		\end{align*}}
	It is easy to verify that $\{\epsilon_t-\epsilon_t^*\}$ is martingale difference sequence and 	
	{\footnotesize\begin{equation}\label{eq-prop}
		\begin{aligned}
		&\mathbb{E}\left[\left\|\frac{1}{\sqrt{k}}\sum_{t=0}^{k-1}\mathbf{D}_{k-1,t}\epsilon_t-\frac{1}{\sqrt{k}}\sum_{t=0}^{k-1}\left(\frac{\mathbf{1}^\intercal}{n}\otimes \mathbf{I}_d\right)\epsilon_t^*\right\|^2\right]\\
		&\le
		2\mathbb{E}\left[\left\|\frac{1}{\sqrt{k}}\sum_{t=0}^{k-1}\mathbf{D}_{k-1,t}\epsilon_t-\frac{1}{\sqrt{k}}\sum_{t=0}^{k-1}\left(\frac{\mathbf{1}^\intercal}{n}\otimes \mathbf{I}_d\right)\epsilon_t\right\|^2\right]\\
		&\quad+2\mathbb{E}\left[\left\|\frac{1}{\sqrt{k}}\sum_{t=0}^{k-1}\left(\frac{\mathbf{1}^\intercal}{n}\otimes \mathbf{I}_d\right)(\epsilon_t-\epsilon_t^*)\right\|^2\right]\\
		&\le
		2\left(\frac{\|u\|\X}{u^\intercal v}\right)^2\frac{1}{1-\tau_\Z^2}U_1\frac{1}{k}+2\mathbb{E}\left[\left\|\frac{1}{\sqrt{k}}\sum_{t=0}^{k-1}\left(\frac{\mathbf{1}^\intercal}{n}\otimes \mathbf{I}_d\right)(\epsilon_t-\epsilon_t^*)\right\|^2\right]\\
		&\le \left(\frac{\|u\|\X}{u^\intercal v}\right)^2\frac{1}{1-\tau_\Z^2}U_1\frac{1}{k}+\frac{2}{k}\sum_{t=0}^{k-1}\mathcal{O}\left(\alpha_t+\alpha_t^2\right)\rightarrow 0,
		\end{aligned}
		\end{equation}}
	where the second inequality follows from the similar analysis of (\ref{noi-consensus}), the third inequality follows from  the fact
	{\small$$\mathbb{E}\left[\|\epsilon_t-\epsilon_t^*\|^2\right]\le 4L^2\mathbb{E}\left[\|x_t-\mathbf{1}\otimes\bar{x}_t\|^2\right]+4nL^2\mathbb{E}\left[\|\bar{x}_t-x^*\|^2\right],$$}
	Lemma \ref{lem:rate} and Theorem \ref{thm:rate}. Define $\bar{\epsilon}_t^*\define1/n\sum_{j=1}^n\epsilon_{j,t}^*$, where $\epsilon_{j,t}^*\define \nabla g_j(x^*;\zeta_{j,t})-\nabla f_j(x^*)$.
	Then by (\ref{eq-prop}) and Slutsky's theorem, (\ref{aux-asym-norm}) is equivalent to
	\begin{equation*}
	\frac{1}{\sqrt{k}}\sum_{t=0}^{k-1}\left(\frac{\mathbf{1}^\intercal}{n}\otimes \mathbf{I}_d\right)\epsilon_t^*=\frac{1}{\sqrt{k}}\sum_{t=0}^{k-1}\bar{\epsilon}_t^*\stackrel{d}{\rightarrow} N\left(\mathbf{0},\frac{1}{n^2}\mathbf{S}\right).
	\end{equation*}
	Denote
	\begin{equation*}
	\xi_{k,t}=\frac{\bar{\epsilon}_t^*}{\sqrt{k}},\quad
	\mathbf{S}_{k,t}=\mathbb{E}\left[\xi_{k,t}\xi_{k,t}^\intercal\right]
	\end{equation*}
	and 
	$$\mathbf{R}_{k,t}=\mathbb{E}\left[\xi_{k,t}\xi_{k,t}^\intercal\bigg|\xi_{k,0},\cdots,\xi_{k,t-1}\right],\quad \mathbf{S}_{k}=\sum_{t=0}^{k-1}\mathbf{S}_{k,t}.$$
	Then Lemma \ref{lem:noi-asym-norm} falls into the setting of \cite[Lemma 3.3.1]{chen2006stochastic}. We just need to verify the conditions of \cite[Lemma 3.3.1]{chen2006stochastic}.
	
	Since $\{\bar{\epsilon}_t^*\}$ is a martingale difference sequence,
	\begin{equation*}
	\mathbb{E}\left[\xi_{k,t}\bigg|\xi_{k,0},\cdots,\xi_{k,t-1}\right]=0,
	\end{equation*}
	which implies the condition (3.3.1) of \cite[Lemma 3.3.1]{chen2006stochastic}.
	
	Next, we verify the conditions (3.3.2)-(3.3.3) of \cite[Lemma 3.3.1]{chen2006stochastic}. By the definition of $\xi_{k,t}$,
	\begin{equation}\label{b-0}
	\begin{aligned}
	\mathbb{E}\left[\left\|\xi_{k,t}\right\|^p\right]&=\mathbb{E}\left[\left\|\frac{\frac{1}{n}\sum_{j=1}^n\nabla g_j(x^*;\zeta_{j,t})}{\sqrt{k}}\right\|^p\right]\\
	&\le\frac{\sum_{j=1}^n\mathbb{E}\left[\|\nabla g_j(x^*;\zeta_{j,t})\|^p\right]}{nk^{p/2}}
	\le\frac{c_f^{p/2}}{k^{p/2}},
	\end{aligned}
	\end{equation}
	where $c_f=\left(1/n\sum_{j=1}^n\mathbb{E}\left[\|\nabla g_j(x^*;\zeta_{j,t})\|^p\right]\right)^{2/p}<\infty$. Then,
	{\small\begin{equation}\label{ieq-noise-bound 1}
		\sup_{k\ge 1}\sum_{t=0}^{k-1}\mathbb{E}\left[\left\|\xi_{k,t}\right\|^2\right]\le \sup_{k\ge 1}\sum_{t=0}^{k-1}\left(\mathbb{E}\left[\left\|\xi_{k,t}\right\|^p\right]\right)^{2/p}\le\sup_{k\ge 1}\frac{kc_f}{k}=c_f.
		\end{equation}
	}

	Note that $\mathbf{S}_{k,t}=\frac{1}{k}\frac{1}{n^2}\Cov\left(\sum_{j=1}^n\nabla g_j(x^*;\zeta_{j})\right)=\frac{1}{k}\frac{1}{n^2}\mathbf{S}$,
	\begin{equation}\label{limit-matrix}
	\mathbf{S}_{k}=\sum_{t=0}^{k-1}\mathbf{S}_{k,t}=\frac{1}{n^2}\mathbf{S}.
	\end{equation}
	(\ref{ieq-noise-bound 1}) and (\ref{limit-matrix}) imply the condition (3.3.2) of \cite[Lemma 3.3.1]{chen2006stochastic}. Moreover, the fact $\mathbf{R}_{k,t}=\mathbf{S}_{k,t}$ almost surely implies the condition (3.3.3) of \cite[Lemma 3.3.1]{chen2006stochastic} directly.
	
	It is left to verify the condition (3.3.4) of \cite[Lemma 3.3.1]{chen2006stochastic}. For any $\delta>0$, by the H\"{o}lder inequality,
	\begin{equation*}
	\begin{aligned}
	&\mathbb{E}\left[\left\|\xi_{k,t}\right\|^21_{\{\left\|\xi_{k,t}\right\|\ge \delta \}}\right]\\
	&\le\left(\mathbb{E}\left[\left\|\xi_{k,t}\right\|^{2(p/2)}\right]\right)^{2/p}\left(\mathbb{E}\left[1^q_{\{\left\|\xi_{k,t}\right\|>\delta\}}\right]\right)^{1/q}\\
	&=\left(\mathbb{E}\left[\left\|\xi_{k,t}\right\|^{p}\right]\right)^{2/p}\P^{1/q}\left(\left\|\xi_{k,t}\right\|>\delta\right)\\
	&\le \left(\mathbb{E}\left[\left\|\xi_{k,t}\right\|^{p}\right]\right)^{2/p}\left(\frac{\mathbb{E}\left[\left\|\xi_{k,t}\right\|\right]}{\delta}\right)^{1/q}\\
	&\le \frac{c^{1+1/(2q)}_f}{k^{1+1/(2q)}\delta^{1/q}},
	\end{aligned}
	\end{equation*}
	where $1_{\mathcal{X}}$ denotes the characteristic function of set $\mathcal{X}$ , which
		means that it equals 1 if $x\in\mathcal{X}$, and 0 otherwie, $p$ is defined in Assumption \ref{ass:objective}, $q$ is the constant satisfying $2/p+1/q=1$, the second inequality follows from Markov inequality, the third inequality follows from (\ref{b-0}). Then
	\begin{equation*}
	\lim_{k\rightarrow\infty}\sum_{t=1}^k\mathbb{E}\left[\left\|\xi_{k,t}\right\|^21_{\{\left\|\xi_{k,t}\right\|\ge \delta \}}\right]\le \lim_{k\rightarrow\infty}
	\frac{kc^{1+1/(2q)}_f}{k^{1+1/(2q)}\delta^{1/q}}=0,
	\end{equation*}
	which implies the condition (3.3.4) of \cite[Lemma 3.3.1]{chen2006stochastic}.
	
	Summarizing above, all the conditions of \cite[Lemma 3.3.1]{chen2006stochastic} hold, then
	\begin{equation*}
	\frac{1}{\sqrt{k}}\sum_{t=0}^{k-1}\mu_{t}^{(0)}\stackrel{d}{\rightarrow} N\left(\mathbf{0},\frac{1}{n^2}\mathbf{S}\right).
	\end{equation*}
	
	In what follows, we show (\ref{ieq-nois-bound-1}). Note that
	{\small\begin{equation*}
		\begin{aligned}
		&\mathbb{E}\left[ \mu_{t} \mu_{s+t}^\intercal\right]\\
		&=\mathbb{E}\left[ \frac{n}{u^\intercal v}\left(\frac{u^\intercal}{n}\otimes\mathbf{I}_{d}\right)\xi_{t}\left( \frac{n}{u^\intercal v}\left(\frac{u^\intercal}{n}\otimes\mathbf{I}_{d}\right)\xi_{s+t}\right)^\intercal\right]\\
		&=\left(\frac{n}{u^\intercal v}\right)^2\left(\frac{u^\intercal}{n}\otimes\mathbf{I}_{d}\right)\mathbb{E}\left[ \left(\sum_{l=0}^{t}\tilde{\mathbf{B}}(t,l)\epsilon_l\right)\right.\\ &\quad\left.\left(\sum_{l=0}^{s+t}\tilde{\mathbf{B}}(s+t,l)\epsilon_l\right)^\intercal\right]\left(\frac{u^\intercal}{n}\otimes\mathbf{I}_{d}\right)^\intercal\\
		&=\left(\frac{n}{u^\intercal v}\right)^2\sum_{l=0}^{t\wedge(s+t)}\left(\frac{u^\intercal}{n}\otimes\mathbf{I}_{d}\right)\tilde{\mathbf{B}}(t,l)\mathbb{E}\left[\epsilon_{l}\epsilon_{l}^\intercal \right]\tilde{\mathbf{B}}(s+t,l)^\intercal\left(\frac{u^\intercal}{n}\otimes\mathbf{I}_{d}\right)^\intercal,
		\end{aligned}
		\end{equation*}}
	where the third equality follows from the property of martingale difference $\{\epsilon_{t}\}$, notation $a\wedge b\define\min\{a,b\}$, $\tilde{\mathbf{B}}(t,l)$ is defined in (\ref{mat-B}). Then we have
	\begin{equation*}
	\begin{aligned}
	&\left\|\mathbb{E}\left[ \mu_{t} \mu_{s+t}^\intercal\right]\right\|\\
	&\le \left(\frac{n}{u^\intercal v}\right)^2\sum_{l=0}^{t\wedge(s+t)}\left\|\left(\frac{u^\intercal}{n}\otimes\mathbf{I}_{d}\right)\tilde{\mathbf{B}}(t,l)\mathbb{E}\left[\epsilon_{l}\epsilon_{l}^\intercal \right]\right.\\
	&\quad\left.\tilde{\mathbf{B}}(s+t,l)^\intercal\left(\frac{u^\intercal}{n}\otimes\mathbf{I}_{d}\right)^\intercal\right\|\\
	&\le \left(\frac{\|u\|}{u^\intercal v}\right)^2\X^2\sum_{l=0}^{t\wedge(s+t)}\left\|\tilde{\mathbf{B}}(t,l)\right\|_\Z\left\|\tilde{\mathbf{B}}(s+t,l)\right\|_\Z\mathbb{E}\left[\left\|\epsilon_{l} \right\|_\Z^2\right]\\
	&\le \left(\frac{\|u\|}{u^\intercal v}\right)^2c_b^2\sum_{l=0}^{t\wedge(s+t)}\tau_\Z^{s+2t-2l}\mathbb{E}\left[\left\|\epsilon_{l} \right\|^2\right]\\
	&\le \left(\frac{\|u\|}{u^\intercal v}\right)^2c_b^2\sum_{l=0}^{t\wedge(s+t)}\tau_\Z^{s+2t-2l}nU_1,
	\end{aligned}
	\end{equation*}
	where $c_b=\max\left\{\X^2,\frac{\left\|\mathbf{B}-\mathbf{I}_{n}\right\|}{\tau_\Z}\X^2\right\}$, the third inequality follows from (\ref{mat-dimish}), the last inequality follows from Assumption \ref{as-conv-1} (iii).
	Then
	{\small\begin{equation*}
		\begin{aligned}
		\sum_{s\ge -t}^{\infty}\left\|\mathbb{E}\left[ \mu_{t} \mu_{s+t}^\intercal\right]\right\|
		&=\sum_{l=0}^{t}\left\|\mathbb{E}\left[ \mu_{t} \mu_{l}^\intercal\right]\right\|+\sum_{l>t}^{\infty}\left\|\mathbb{E}\left[ \mu_{t} \mu_{l}^\intercal\right]\right\|\\
		&\le \left(\frac{\|u\|}{u^\intercal v}\right)^2 c_b^2nU_1\sum_{l=0}^{t}\sum_{r=0}^{l}\tau_\Z^{l+t-2r}\\
		&\quad+\left(\frac{\|u\|}{u^\intercal v}\right)^2 c_b^2nU_1\sum_{l>t}^{\infty}\sum_{r=0}^{t}\tau_\Z^{t+l-2r} \\
		&\le \left(\frac{\|u\|}{u^\intercal v}\right)^2 c_b^2nU_1\left(\sum_{l=0}^{t}\frac{\tau_\Z^{t-l}}{1-\tau_\Z}+\sum_{l>t}^{\infty}\frac{\tau_\Z^{l-t}}{1-\tau_\Z}\right)\\
		&\le \left(\frac{\|u\|}{u^\intercal v}\right)^2 c_b^2nU_1\frac{1+\tau_\Z}{\left(1-\tau_\Z\right)^2},
		\end{aligned}
		\end{equation*}}
	which implies (\ref{ieq-nois-bound-1}). The proof is complete.
\end{proof}

\ifCLASSOPTIONcaptionsoff
  \newpage
\fi



%



\end{document}